\documentclass{amsart}

\usepackage{amssymb, hyperref, color}
\usepackage{graphicx}
\usepackage[all]{xypic}
\usepackage{verbatim}
\usepackage{tikz}
\usepackage{lipsum}
\usepackage{amsfonts}
\usepackage{graphicx}
\usepackage{epstopdf}
\usepackage{algorithmic}

\usepackage[margin=3.5cm]{geometry}

\ifpdf
  \DeclareGraphicsExtensions{.eps,.pdf,.png,.jpg}
\else
  \DeclareGraphicsExtensions{.eps}
\fi

\usepackage{pgfplots}
\usepgfplotslibrary{fillbetween}

\newtheorem{theorem}{Theorem}
\numberwithin{theorem}{section}
\newtheorem{proposition}[theorem]{Proposition}
\newtheorem{lemma}[theorem]{Lemma}
\newtheorem{corollary}[theorem]{Corollary}

\newtheorem{example}[theorem]{Example}

\newcommand{\RR}{\mathbb{R}}

\newcommand{\CC}{\mathbb{C}}

\newcommand{\KK}{\mathbb{K}}

\newcommand{\T}{^\mathsf{T}}

\DeclareMathOperator{\rank}{rank}

\title{Learning Paths from Signature Tensors}

\author{Max Pfeffer, Anna Seigal, and Bernd Sturmfels}

\usepackage{amsopn}



\begin{document}

\maketitle

\begin{abstract}
Matrix congruence extends naturally to the setting of tensors. 
 We apply methods from tensor decomposition,  algebraic geometry and numerical optimization 
to this group action.
Given a tensor in the orbit of another tensor, we 
compute a matrix which transforms one to the other.
Our primary application is an inverse problem
from  stochastic analysis:
the recovery of paths from their third order signature tensors.
We establish identifiability results, both exact and numerical, for
piecewise linear paths, polynomial paths, and generic dictionaries.
Numerical optimization is applied for recovery from inexact data.
We also compute the shortest path with a given signature tensor.
\end{abstract}

%
%

%

\section{Introduction}

In many areas of applied mathematics, tensors are used to encode
features of geometric data. The tensors then serve as the input
to algorithms aimed at classifying and
understanding the original data.  This set-up comes with a 
natural inverse problem, namely to recover the geometric
objects from the tensors that represent them.
The aim of this article is to solve this inverse problem.

Our motivation comes from the {\em signature method in machine learning} \cite{CK}.
In this setting, the geometric object is a path $[0, 1] \to \RR^d$.
 The path is encoded by its signature, an infinite sequence of
tensors that are interrelated through a Lie algebra structure.
Signature tensors were introduced by Chen~\cite{C}, and they
play an important role in stochastic analysis \cite{FH, LQ02}.
We refer to \cite{LX2, LX1} for the recovery problem,
and to \cite{DR, Gal, KSHGL,KSHGL2} for algorithms and applications.
Our point of departure is the approach to signature tensors via algebraic geometry 
that was proposed in~\cite{AFS}.

The problem we address is path recovery from the {\em signature tensor of order three}.
This tensor is the third term in the signature sequence.
Higher order signature tensors encode finer representations of a path than lower order signatures:
if two paths (which are not loops) agree at the signature tensor of some order, 
then they agree 
up to scale
at all lower orders \cite[Section 6]{AFS}.
We focus on order three because,
like in many similar contexts~\cite{KB}, 
tensors under the congruence action have useful 
uniqueness properties that do not hold for matrices.
The full space of paths is too detailed for meaningful recovery from finitely many
numbers. As is often done~\cite{LX2}, we restrict to paths which lie in a particular family.
 We consider paths whose coordinates can be written as linear combinations
  of functions in a fixed {\em dictionary}. 
The dictionary determines a core tensor, which is transformed 
by the congruence action into signatures of paths in the family.
The family of piecewise linear paths is a main example.

This article is organized as follows. In Section \ref{sec2} we describe our set-up
which emphasizes the notion of a dictionary to describe a family of paths. 
Our main contributions begin in 
Section \ref{newsec}, where we investigate tensors under the congruence action by matrices. We obtain necessary and sufficient conditions for the size of the stabilizer under this group action.
This leads to conditions under which a 
path can be recovered uniquely, or up to a finite list of choices, from the third order signature tensor, in Section~\ref{sec3}.
We then apply these conditions to give identifiability results for 
generic dictionaries, in Section \ref{sec4}, and piecewise linear paths, in Section \ref{sec5}, where
our results prove part of \cite[Conjecture 6.10]{AFS}.
In Section \ref{sec7} we study numerical identifiability for the recovery of paths from signature data. 
Both upper bounds and lower bounds
are given for the numerical non-identifiability, the scaled inverse distance to the set of instances where the recovery problem is ill-posed. 
In Section \ref{sec8} we turn to numerical optimization and 
we present experimental results on unique
recovery of low-complexity paths.
Section \ref{sec10} addresses the problem of
 finding the shortest path with given third signature tensor.

\section{Dictionaries and their Core Tensors}
\label{sec2}

We fix a {\em dictionary} $\psi = (\psi_1, \psi_2,$
$ \ldots, \psi_m)$ of piecewise differentiable functions 
$\psi_i : [0,1] \rightarrow \RR$. The dictionary corresponds to a path
in $\RR^m$, also denoted $\psi$, whose $i$th coordinate is $\psi_i$. The path~$\psi$ is 
regarded as a fixed reference path in $\RR^m$.
Its signature is a formal series of tensors
$$ \sigma(\psi)\,\,= \,\,\sum_{k=1}^\infty \sigma^{(k)}(\psi) , $$
whose $k$th term is a tensor in $(\RR^m)^{\otimes k}$ with
entries that are iterated integrals of $\psi$:
\begin{equation}
\label{eq:generalintegral}
 (\sigma^{(k)}(\psi))_{i_1 i_2 \cdots i_k} \,\,= \,\, \int_0^1 \cdots \left( \int_{0}^{t_3}   \left( \int_{0}^{t_2} {\rm d} \psi_{i_1}(t_1) \right) \,{\rm d} \psi_{i_2}(t_2) \right) \cdots {\rm d} \psi_{i_k}(t_k). 
 \end{equation}

Evaluating~\eqref{eq:generalintegral} for $k=1$ shows that the first signature $\sigma^{(1)}(\psi)$ is the vector~$\psi(1)-\psi(0)$. The second signature $\sigma^{(2)}(\psi)$ is the matrix
$\frac{1}{2}(\psi(1)-\psi(0))^{\otimes 2} + Q$,
where $Q$ is skew-symmetric. Its entry $q_{ij}$ is the {\em L\'evy area}
of the projection of $\psi$ onto the plane indexed by $i$ and~$j$,
the signed area between the planar path and the segment connecting its endpoints.
For background on signature tensors of paths and their applications
 see \cite{AFS, CK, DR, FH, LQ02, LX2,LX1}.

\smallskip

This article is based on the following two premises:
\smallskip

\begin{itemize} 
\item[(a)] We study the images of a fixed reference path $\psi$ under linear maps. 
\item[(b)] We focus on the third order signature $ \sigma^{(3)}(\psi)$.
\end{itemize}
\smallskip

\noindent We first discuss premise (a). Consider a linear map $\RR^m \rightarrow \RR^d$ given by a $d \times m$ matrix $X = (x_{ij})$.
The image of the path $\psi$ under $X$ is the 
path $\,X \psi \,:\,[0,1]\, \rightarrow \,\RR^d\,$ given by
\begin{equation*}
\label{eq:pathX}
 t  \,\mapsto\,
\biggl( \, \sum_{j=1}^m x_{1j} \psi_j(t) \,, 
  \sum_{j=1}^m x_{2j} \psi_j(t) \,, 
\, \ldots\,,\, \sum_{j=1}^m x_{dj} \psi_j(t) \,\biggr).
\end{equation*}
The following key lemma relates the linear transformation of a path to
the induced linear transformation of its signature tensor.
The proof follows directly from the iterated integrals in~\eqref{eq:generalintegral}, 
bearing in mind that integration is a linear operation.

\begin{lemma} \label{lem:equivariant}
The signature map is equivariant under linear transformations, i.e.
\begin{equation}
\label{eq:equivariant}
\sigma(X \psi) \,\, = \,\,  X (\sigma(\psi)).
\end{equation}
\end{lemma}

The action of the linear map $X$ on the signature $\sigma(\psi)$ is as follows. 
The~$k$th order signature of $\psi$ is a tensor in $(\RR^m)^{\otimes k}$.
We multiply this tensor on all $k$ sides
 by the $d \times m$ matrix $X$. The result of this
tensor-matrix product, which is known as
multilinear multiplication, is a tensor in $(\RR^d)^{\otimes k}$. Using notation from the
theory of tensor decomposition \cite{KB}, 
the identity \eqref{eq:equivariant} can be written as
\begin{equation}
\label{eq:equivariant2}
 \quad \sigma^{(k)}(X \psi) \,\, = \,\,  [\![\,\sigma^{(k)}(\psi)\,;\, X, X, \ldots, X \,]\!] 
\qquad \hbox{for $\,k=1,2,3,\ldots$} 
\end{equation}
For $k =1$ this is the matrix-vector product
$\,\sigma^{(1)} (X \psi)  = X \cdot \sigma^{(1)}(\psi)$.
For $k = 2$ the rectangular matrix $X$ acts  on the square signature matrix via the congruence action:
$$ \sigma^{(2)}(X\psi) \,\, = \,\, X \cdot \sigma^{(2)}(\psi) \cdot X\T. $$
Lemma~\ref{lem:equivariant} means that, once
the signature of the dictionary $\psi$ is known, integrals no longer need to be computed.
Signature tensors of a path arising from $\psi$ by a linear transformation
are obtained by tensor-matrix multiplication. This works for
many useful families of paths.

We next justify our premise (b).
For a path in $\RR^d$, the number of entries in
 the $k$th signature tensor is $d^k$. For $k \geq 4$, this quickly becomes prohibitive.
  But $k=2$ is too small: signature matrices do not contain enough
information to recover paths in a meaningful way.
Even after fixing all $\binom{d}{2}$ L\'evy areas,
there are too many paths between two  points in $\RR^d$. 
The third signature is the right compromise.
The number $d^3$ of entries is reasonable,
paths are identifiable from third signatures under certain conditions,
and we propose practical algorithms for path recovery.
This paper establishes the last two points, assuming premise~(a).

With our two premises in mind, we give
more detail for the case of third signature tensors, $k=3$. We fix a dictionary $\psi$ consisting of $m$ functions.
We refer to its third signature $C_\psi = \sigma^{(3)}(\psi) \in \RR^{m \times m \times m}$
as the {\em core tensor} of $\psi$. 
This tensor has entries
\begin{equation}
\label{eq:tripleintegral} \quad
 c_{ijk} \,\,= \,\, \int_0^1 \! \left( \int_{0}^{t_3} \left( \int_{0}^{t_2} 
 {\rm d} \psi_i(t_1) \right) {\rm d} \psi_j(t_2) \right) {\rm d} \psi_k(t_3)
 \quad \hbox{for all} \,\, 1 \leq i,j,k \leq m .
 \end{equation}
 The first and second signature of a real path are determined
 by the third signature, provided the path is not a loop, just as any 
 lower order 
 signature tensor can be recovered up to scale from higher order signatures.
This follows from the {\em shuffle relations} \cite[Lemma 4.2]{AFS}.
Writing $c_i$ and $c_{ij}$ for the entries of the first and second signature
respectively, we have
the identities
\begin{equation}
\label{eq:shufflerel} 
 c_i c_j \,=\, c_{ij} + c_{ji}
\quad {\rm and} \quad
c_i c_{jk} \,=\, c_{ijk} +  c_{jik} + c_{jki} .
\end{equation}
 
Given a $d \times m$ matrix $X = (x_{ij})$, the third signature of the image path $X \psi $ in $\RR^d$ is denoted by $\sigma^{(3)}(X)$, as shorthand for $\sigma^{(3)}(X \psi)$.
Following \eqref{eq:equivariant2},
this $d \times d \times d$ tensor is obtained from $C_{\psi} = (c_{ijk})$ by multiplying 
by $X$ on each side. The entry of 
$\sigma^{(3)}(X)$ in position $(\alpha, \beta, \gamma)$~is
 \begin{equation}
\label{eq:action}
[\![\, C_\psi \,;\, X , X , X \,]\!]_{\alpha \beta \gamma} 
\,\,\,=\,\,\, \sum_{i=1}^m \sum_{j=1}^m \sum_{k=1}^m c_{ijk} x_{\alpha i} x_{\beta j} x_{\gamma k} .  
\end{equation}
The expression \eqref{eq:action} for the signature tensor in terms of
the core tensor is closely related to the {\em Tucker decomposition} \cite{KB, T}
 that arises frequently in tensor compression.
 Our application differs from the usual setting 
 in that the core tensor has fixed size and fixed entries.
Furthermore, we multiply each side by the same matrix.
     
\smallskip
  
We next discuss some specific dictionaries and the paths they encode, starting with two dictionaries studied in~\cite{AFS}.
The first dictionary is $\psi(t) = (t,t^2,\ldots, t^m)$. Multiplying this  dictionary by matrices $X$ of size $d \times m$ gives all {\em polynomial paths} of degree at most $m$ that start at the origin  in $\RR^d$.
 The core tensor of $\psi$ is denoted by $C_{\rm mono}$ to indicate the monomials  $t^i$.
  By~\cite[Example 2.2]{AFS}, its entries are
\begin{equation} \label{mono}
c_{ijk} \,\,=\,\, \frac{j}{i+j} \cdot \frac{k}{i+j+k}.
\end{equation}

Our second dictionary comes from an axis path in $\RR^m$. It
encodes all {\em piecewise linear paths} with $\leq m$ steps.
The $i$th entry in the dictionary is the piecewise linear basis function
\begin{equation}
\label{PLpath}
\psi_i(t) \,\,  = \,\, \begin{cases} \qquad 0 & {\rm if} \,\, \, t \leq \frac{i-1}{m}, \\
mt - (i-1) & {\rm if} \,\, \frac{i-1}{m} < t < \frac{i}{m}, \\
\qquad 1  & {\rm if} \,\,\,t \geq \frac{i}{m} .
\end{cases}
\end{equation}
By \cite[Example 2.1]{AFS}, the associated core tensor $C_{\rm axis}$ 
is ``upper-triangular'', namely
\begin{equation} \label{axis} c_{ijk} \quad = \quad \begin{cases}
\,\, 1  &{\rm if} \,\,\, i < j < k  ,\\
\,\, \frac{1}{2} & {\rm if} \,\,i < j=k \,\, {\rm or} \,\, i = j < k, \\
\,\, \frac{1}{6} & {\rm if} \,\,\, i=j=k, \\
\,\, 0 & {\rm otherwise}.
\end{cases}
\end{equation}

The tensors $C_{\rm mono}$ and $C_{\rm axis}$ are real points in the {\em universal variety} $\,\mathcal{U}_{m,3} \subset (\CC^m)^{\otimes 3}$.
This consists of all third order signatures of paths in $\CC^m$, or equivalently
all core tensors of dictionaries of size $m$. At present, we do not know whether all
real points in $\mathcal{U}_{m,3}$ are in the topological closure of the signature tensors of real paths.

\begin{proposition}\label{prop:dimu}
The variety $\,\mathcal{U}_{m,3}$ is irreducible of dimension
$\,\frac13 m^3 + \frac12 m^2 + \frac16 m$.
\end{proposition}

\begin{proof}
This follows from \cite[Theorem 6.1]{AFS}. Note that, in the present paper,  
$\mathcal{U}_{m,3}$ denotes the affine variety,  whereas \cite{AFS} refers to projective varieties.
The dimension of $\,\mathcal{U}_{m,3}$ is the number of Lyndon words (see~\cite{Reut})
on $m$ letters of length $1$, $2$ or $3$. These three numbers are
$m$,  $\binom{m}{2}$ and $ \frac{1}{3}(m^3-m)$.
Their sum equals $\frac13 m^3 + \frac12 m^2 + \frac16 m $.
\end{proof}

The polynomials that define $\mathcal{U}_{m,3}$
are obtained by eliminating the unknowns $c_i$ and $c_{ij}$
from  the equations \eqref{eq:shufflerel}. We provide more details at the end of Section~\ref{sec4}.

\begin{example}[Generic Dictionaries] \rm \label{ex:gendic}
We describe a method for sampling real points in the universal variety $\mathcal{U}_{m,3}$, assuming \cite[Conjecture 6.10]{AFS}.  Pick $M$ random vectors $Y_1,Y_2,\ldots,Y_M$ in $\RR^m$,
where $M$ exceeds  $\,\frac13 m^2 + \frac12 m + \frac16 \,=\, {\rm dim}(\mathcal{U}_{m,3})$,
and take the piecewise linear path with steps $Y_1,Y_2,\ldots,Y_M$.
By \cite[Example 5.4]{AFS}, the resulting {\em generic core tensor} equals
\begin{equation}
\label{eq:sigma3chen}
\! C_{\rm gen}  \, = \,
\frac{1}{6} \cdot \sum_{i=1}^M Y_i^{\otimes 3}\,\,+\,\,
\frac{1}{2} \cdot \!\! \sum_{1 \leq i < j \leq M}\!\! \! \bigl(Y_i^{\otimes 2} \otimes Y_j +
Y_i \otimes Y_j^{\otimes 2} \bigr) \,\,\, + \! \sum_{1 \leq i < j < k \leq M} \!\!\!\!\!\!
Y_i \otimes Y_j \otimes Y_k .
\end{equation}
The coefficients in \eqref{eq:sigma3chen} match the tensor entries in \eqref{axis}.
By Chen's Formula \cite[eqn.~(38)]{AFS}, the signature tensor
 $C_{\rm gen}$ is  the degree $3$ component~in 
the tensor series $\sigma(\psi) =
{\rm exp}(Y_1) \otimes {\rm exp}(Y_2) \otimes \,\cdots\, \otimes\, {\rm exp}(Y_M)$, where $\exp(Y_i) = \sum_{k = 0}^{\infty} \frac{1}{k!} Y_i^{\otimes k}$.

An alternative method for sampling from $\mathcal{U}_{m,3}$
uses the Gr\"obner~basis in \cite[Theorem 4.10]{AFS}.
We write $\sigma_{\rm lyndon}$ for the vector
of all signatures $\sigma_i$, $\sigma_{ij}$
and $\sigma_{ijk}$ whose indices are Lyndon words.
This includes all $m$ first order signatures $\sigma_i$,
all $\binom{m}{2} $ second order signatures $\sigma_{ij}$
with $i < j$, and all 
$\frac{1}{3} (m^3-m)$ third order signatures
$\sigma_{ijk}$ satisfying
$i < {\rm min}(j,k)$ or $i=j < k$.
We pick these $m + \binom{m}{2} + \frac{1}{3}(m^3-m)$
signature values to be random real numbers and
substitute these numbers into the vector $\sigma_{\rm lyndon}$.
The non-Lyndon signatures $\sigma_{ijk}$ are then computed
by evaluating $\phi_{ijk}(\sigma_{\rm lyndon})$, where
$\phi_{ijk}$ is the normal form polynomial  in \cite[Theorem 4.10]{AFS}. 
\end{example}

We now define what we mean by ``learning paths" in the title of this paper.
Let $C$ be a fixed core tensor of format $m \times m \times m$,
such as $C_{\rm axis}$, $C_{\rm mono}$ or $C_{\rm gen}$.
Our data is a $d \times d \times d$ tensor $S = (s_{ijk})$
 that is the third signature of some path in $\RR^d$.
Our hypothesis is that the path can be represented by the dictionary~$\psi$, i.e. it is the image of $\psi$ under a linear map. We seek a $d \times m$ matrix $X = (x_{ij}) $ that satisfies $S = \sigma^{(3)}(X)$.
In other words, given $C$ and $S$, we wish to solve the tensor equation
$$   [\![\,C \,;\, X, X, X \,]\!] \,\,\, = \,\,\, S. $$
This is the system of $d^3$ cubic equations in $md$ unknowns $x_{ij}$
 \begin{equation}
\label{eq:action2}
\qquad
 \sum_{i=1}^m \sum_{j=1}^m \sum_{k=1}^m c_{ijk} x_{\alpha i} x_{\beta j} x_{\gamma k} \,\, = \,\,
 s_{\alpha \beta \gamma} \qquad {\rm for} \,\, 1 \leq \alpha, \beta, \gamma \leq d. 
\end{equation}
The system \eqref{eq:action2} has a solution $X$ if and only if 
the dictionary with core tensor $C$ admits a path with signature tensor $S$.
For the dictionaries we consider, the solution $X$ is conjectured to be unique among real matrices $X$ provided $m <
\frac{1}{3} d^2 + \frac{1}{2} d + \frac{1}{6}$, 
and unique up to scaling by a third root of unity if we allow complex matrices $X$.
The inequality means that
the dimension of the universal variety $\mathcal{U}_{d,3}$ exceeds
the number $md$ of unknowns, which is necessary for identifiability.
For piecewise linear and polynomial paths, this 
is presented in Conjecture 6.12 and Lemma 6.16 of \cite{AFS}.

 We experimented with  {\em Gr\"obner bases} for 
 solving the equations $ [\![C; X,X,X ]\!] = S$, 
 thereby extending the computations in~\cite[Section 6]{AFS}.
Table \ref{tab:grobner1} summarizes our findings for various values of $d$ and $m$.
The number $md$ is shown with lower index $d^3$.
After the hyphen we report timings (in seconds) for the following computation.
We pick a  $d \times m$ matrix $X_0$ with random integer entries, sampled uniformly 
between $-15$ and $15$, and  we consider the system of equations 
$\,  [\![C; X,X,X ]\!] \, = \,[\![C; X_0,X_0,X_0]\!]$.
We computed two Gr\"obner bases, the first for
$C = C_{\rm axis}$ and the second for $C = C_{\rm mono}$.
We did this experiment in maple 16, using
the command {\tt Basis}  with the default order (degree reverse lexicographic)
in the {\tt Groebner} package. The results are shown in Table~\ref{tab:grobner1}.
In all cases for which the computation succeeded, there were
three standard monomials, corresponding to the matrices 
 $\eta X_0$, where $\eta^3  = 1$. This confirms rational identifiability
for the core tensor.

\begin{table}[h] \qquad 
\begin{center}
\begin{tabular}{ | l | r | r | r | r | r | r | r | r |} \hline  $m \backslash d$  & 2 & 3 & 4 & 5  & 6  & 7  \\ 
\hline   2 & $\!\!4_8 - 0,0 $ &$6_{27} - 0,0 $     & $8_{64} - 0,0$     & $10_{125} - 0,0$  & $12_{216} - 0,1$  & $14_{343}-0,1$ \\
\hline   3 & $6_8 - {\rm NI \,\,}$   &$9_{27} - 0,0$     & $12_{64} - 0,0$   & $15_{125} - 0,1$  &  $18_{216} - 0,1$  & $21_{343}-1,3 $  \\
\hline   4 & $8_8 - {\rm NI \,\,} $ & $12_{27} - 0,0$  &  $16_{64}-0,1$  & $20_{125} - 1,1$ &  $24_{216} -2,4 $ & $28_{343} - 5,9$   \\
\hline   5 & $10_8 - {\rm NI \,\,}$ & $15_{27} - {\rm N I \,\,}$  & $20_{64} - 4,22 $ & $25_{125} {-} 16,43$ 
& $30_{216} {-} 72,188 $ & $35_{343}{-}601 ,{\rm F} $  \\
\hline 
 \end{tabular} 
 \end{center}
 \caption{\noindent Using Gr\"obner bases to 
   recover a path from its third signature. The first two numbers in each box are the size
  of the problem:
the first number counts the unknowns; its index counts the number of  equations. 
  The next two numbers are timings for  Gr\"obner basis computations 
 in {maple~16}.   The first entry
is for $C_{\rm axis}$, the second for $C_{\rm mono}$. The units are seconds, rounded down.
 An entry 
 NI means that the model is not identifiable, while F  means that the computation failed to terminate.
 \label{tab:grobner1} 
}
 \end{table}
 
We conclude this section by
  mentioning two group actions, closely related to ours, which have been studied extensively.
 The first  concerns homogeneous polynomials
  $f(x_1, \ldots, x_m)$. Matrices
 $Z $ in ${\rm GL}(m,\CC)$ act by linear change of variables~$\,f(x) \mapsto f( Z \cdot x)$.
 This is precisely our congruence
 action  $C \mapsto [\![C; Z, Z, Z]\!]$ in the special case 
 where $C$ is a {\em symmetric} tensor that corresponds~to the cubic polynomial
 $$ f(x_1, \ldots, x_m) \,\,= \sum_{i,j,k=1}^m c_{ijk} x_i x_j x_k .$$
%
 Another well-studied  action (cf.~\cite{BGOWW,Lan}) concerns tensors of any
    size $m_1 \times \cdots \times m_k$. The group ${\rm GL}(m_1,\CC) \times \cdots \times 
    {\rm GL}(m_k,\CC)$
    acts via   $C \mapsto [\![C; Z_1, \ldots, Z_k ]\!]$ where $Z_i \in {\rm GL}({m_i},\CC)$.
For $k=3$ this is $C \mapsto [\![C; Z_1, Z_2, Z_3]\!]$.
Our action is the restriction to the diagonal $Z:= Z_1 = Z_2 = Z_3$.
 There is a literature on the above two group actions, but much less on 
 the {\em congruence action} of ${\rm GL}(m,\CC)$ on $(\CC^{m})^{\otimes k}$
 which is needed here.
 
\section{Stabilizers under Congruence}
\label{newsec}
 
 From now on, the letter $\KK$ denotes a field, usually 
 either the real numbers $\RR$ or the complex numbers $\CC$.
We study the congruence action of invertible matrices $X\in \KK^{m \times m}$ on
the space of tensors $T \in (\KK^{m})^{\otimes k}$ via
$$T \,\,\mapsto \,\,[\![T; X, X, \ldots, X]\!] . $$
Writing $X = (x_{ij})$, $T = (t_{\alpha_1\ldots\alpha_k})$, the entries of the transformed tensor
are
$$ [\![T; X, X, \ldots, X]\!]_{\beta_1 \ldots \beta_k} \quad
= \,\,\,\sum_{\alpha_1,\ldots,\alpha_k} t_{\alpha_1 \ldots \alpha_k} x_{\beta_1 \alpha_1} x_{\beta _2 \alpha_2} \cdots x_{\beta_k \alpha_k} .$$
This tensor is the image of $T$ under the congruence action by $X$.
The {\em stabilizer} of $T$ under the group action is the subgroup of matrices $X$ 
in ${\rm GL}(m,\KK)$ that satisfy
$[\![ T ;X, X, \ldots, X ]\!]  = T$. We denote it by
 ${\rm Stab}_\KK(T)$. 
  The stabilizer is defined by a system of polynomial equations of degree~$k$ in 
  the entries of $X$.
Matrices $\eta I$ with $\eta^k = 1$ are always among the solutions.

Section \ref{sec3} will relate the stabilizer of $T$ under the congruence action to the identifiability of path recovery within the family of paths whose dictionary has signature tensor $T$.
It is an open problem to characterize tensors $T$  in $(\KK^m)^{\otimes k}$
whose stabilizer under congruence is {\em non-trivial}, i.e. for which~${\rm Stab}_\KK(T) $
strictly contains $ \{ \eta I  : \eta^k=1 \}$.

We introduce an important notion for stabilizers under congruence, which we call {\em symmetrically concise}.
It means that, for $T \in (\KK^m)^{\otimes k}$, there is no subspace $W \subsetneq \KK^m$ such that
$ T \in W^{\otimes k}$.

We can define symmetrically concise in terms of flattenings. 
The tensor $T$ has $m^k$ entries and~$k$ principal flattenings, matrices of size $m \times m^{k-1}$. The $i$th flattening $T^{(i)}$ has rows labeled by the $i$th index of $T$ and its
columns labeled by a multi-index from all remaining indices~\cite{KB,Lan}. Flattenings 
are also known as matricizations.
We recall from~\cite{Stra} that
 a tensor $T \in (\KK^m)^{\otimes k}$ is {\em concise} if it has flattening ranks $(m,m,$ $\ldots,m)$.
We concatenate the $k$ flattening matrices to form a single matrix of size $m \times km^{k-1}$. 
The tensor is {symmetrically concise} if this matrix has full rank $m$. 

For symmetric tensors concise and symmetrically concise are equivalent, because the $m \times k m^{k-1}$ matrix consists of $k$ identical blocks of size $m \times m^{k-1}$.
However, symmetrically concise is weaker than concise for non-symmetric tensors. For instance, the $3 \times 3 \times 3$
basis  tensor $T = e_1 \otimes e_2 \otimes e_3$ is symmetrically concise but not concise: there exist subspaces $W_i \subsetneq \KK^3$ with $T \in W_1 \otimes W_2 \otimes W_3$, but 
we cannot find the same subspace $W \subsetneq \KK^3$ across all modes such that $T \in W^{\otimes 3}$.
   
We can also define symmetrically concise from a decomposition into rank one terms. 
A tensor $T \in (\KK^{m})^{\otimes k}$ is {\em rank one} if 
$T = v^{(1)} \otimes v^{(2)} \otimes \cdots \otimes v^{(k)}$ for some non-zero vectors $v^{(j)} \in \KK^m$.
The {\em rank} of  $T$ (over $\KK$) is the minimal number of terms 
in an expression for $T$ as a sum of rank one tensors, $T = \sum_{l = 1}^r T_l$ for $T_l = v_l^{(1)} \otimes v_l^{(2)} \otimes \cdots \otimes v_l^{(k)}$ with $v_l^{(j)} \in \KK^m$. 
We call a decomposition of~$T$ of minimal length a {\em minimal decomposition}. A tensor is symmetrically concise if the $kr$ vectors $v_l^{(j)}$ in any minimal decomposition span the ambient space $\KK^m$.

\begin{proposition}\label{prop:kers}
Let $T \in (\KK^{m})^{\otimes k}$ be a tensor that is not symmetrically concise. Then the stabilizer of $T$ under the congruence action is non-trivial.
\end{proposition}

\begin{proof}
Since $T$ is not symmetrically concise, there exists a vector $v \in \KK^m$ of norm one
 such that $v\T T^{(i)} = 0$ for all $i = 1, \ldots, k$. This condition implies 
 $\,[\![ T ; I {+} vv\T, I {+} vv\T,$ $ \ldots, I {+} vv\T ]\!] = T$.
 Hence the invertible matrix $I + vv\T$ is in the stabilizer of $T$.
\end{proof}

Tensors with trivial stabilizer are symmetrically concise but not always concise:

\begin{example} \rm
\label{ex:idcond}
Consider the rank-one tensor $T = e_1 \otimes e_2 \otimes (e_1 + e_2)$.
Each $2 \times 4$ flattening matrix of $T$ is rank-deficient.
This means that $T$ has flattening ranks $(1,1,1)$, so $T$ is not concise.
However, the $2 \times 12$ matrix we obtain by concatenating the three flattening matrices has full rank, hence the tensor $T$ is symmetrically concise. The stabilizer of $T$ is directly computed to be trivial.
\end{example}

We next derive a Jacobian criterion which gives a sufficient condition
for the stabilizer of a tensor under the congruence action to be finite. 
For notational simplicity we state the criterion only for order three tensors. 
The Jacobian $\nabla f(X) \in \KK^{m^3 \times m^2}$
of the function $f(X) = [\![T; X, X, X]\!]$
 has entries:
$$
\nabla f(X)_{(i,j,k),(u,v)}  =  \frac{ \partial f_{ijk} }{\partial x_{uv}}	
=   \sum_{\alpha,\beta} ( t_{v\alpha \beta} \delta_{u i} x_{j\alpha} 
x_{k\beta} + t_{\alpha v \beta} \delta_{u j} x_{i\alpha} x_{k\beta} + t_{\alpha \beta v} \delta_{u k} x_{i\alpha} x_{j\beta} ),
$$
where $\delta_{ij}$ is the Kronecker delta.
The entries of the Jacobian at $X =I$ are
\begin{equation}
\label{eq:idjac}
\nabla f(I)_{(i,j,k),(u,v)} \,\,\,=\,\,\, t_{vjk} \delta_{ui} + t_{ivk} \delta_{uj} + t_{ijv} \delta_{uk}.
\end{equation}
Consider the $m^2 \times m^2$ submatrix $J_1$ of the Jacobian obtained by setting $k = 1$ in~\eqref{eq:idjac}. The entry of $J_1$ in row $(i,j)$ 
 and column $(u,v)$ is the linear form
\begin{align*}
J_1((i,j),(u,v)) &\,\,=\,\, \, \delta_{ui} t_{vj1} + \delta_{uj} t_{iv1} + \delta_{u1} t_{ijv}  .
\end{align*}

\begin{proposition} \label{prop:j1}
Let $\,T$ be a tensor whose $m^2 \times m^2$ matrix $J_1$ as above is invertible.
Then the stabilizer of $\,T$ under the congruence action by ${\rm GL}(m,\KK)$ is finite.
\end{proposition}

\begin{proof}
The stabilizer of $T$ under congruence is infinite when the map $ f : Z \,\mapsto\,\, [\![T; Z, Z, Z]\!]$ has positive-dimensional fibers.
 If the matrix $J_1$ is invertible then the Jacobian $\nabla f$
has full rank at $Z=I$. 
This implies that a connected component of the stabilizer 
consists of the single matrix $I$.
Consider another connected component of the stabilizer, containing a matrix $X$. Applying $X^{-1}$ to the component gives a connected component of the stabilizer containing $I$, which therefore must be the single matrix~$I$. Hence all connected components are zero-dimensional, and the stabilizer is finite. 
\end{proof}

The same conclusion holds if any of the maximal minors of the Jacobian in $\KK^{m^k \times m^2}$ is non-zero.
Proposition~\ref{prop:j1} implies that the tensor $C_{\rm mono}$ with entries~\eqref{mono} has finite stabilizer under the congruence action for $m \leq 30$. For example, for $m=10$ we compute $\,{\rm det}(J_1)^{-1} = 
2^{288}  3^{160}  5^{81} 7^{75}  11^{96}   13^{86}   17^{52}  19^{35}$.
We can use this to show that polynomial paths are algebraically identifiable from their third signature tensors
when $m \leq 30$ and $m \leq d$. This doubles
the bound $m=15$ from \cite[Lemma 6.16]{AFS}, contributing
progress towards the proof of \cite[Conjecture 6.10]{AFS}.
We will study the numerical identifiability of such paths in Section~\ref{sec7}.
 
\section{Criteria for Identifiability}
\label{sec3}

In this section, we relate the stabilizer of a core tensor $C$ under  congruence to 
the identifiability of path recovery from the signature tensor
 $[\![C; X, X, \ldots, X]\!]$, where $X$ is a rectangular matrix in $\KK^{d \times m}$.
  We focus on the case $m \leq d$,
when the dictionary size is smaller than  the dimension of the ambient space.
We shall argue that, to study the identifiability of paths 
from their signatures, it suffices to study  identifiability of $C$ under congruence.
The signature tensor is {\em identifiable} if the matrix $X$ can be recovered up to scale.
The signature tensor is  {\em algebraically identifiable} if $X$ can be recovered up to a finite list of choices.
Finally, it is {\em rationally identifiable} if, up to scale, 
$X$ can be written as a rational function in~$C$.

The following result compares minimal decompositions of smaller tensors with those of larger tensors in which they appear as a block. Any decomposition of the larger tensor is obtained from a decomposition of the smaller tensor by adding zeros.

\begin{lemma} \label{adjoin} Let $T \in (\KK^d)^{\otimes k}$ be a 
 tensor with all entries zero outside a block of size
$m \times m \times \cdots \times m$, for some $m \leq d$. 
Then any rank one term in a minimal decomposition of $T$
is also zero outside of the block.
\end{lemma}

\begin{proof} 
Let $T = \sum_{l = 1}^r v_l^{(1)} \otimes v_l^{(2)} \otimes \cdots \otimes v_l^{(k)}$ be a minimal decomposition. Assume that a rank one term is non-zero outside of the block, i.e.~the coordinate
 $v_l^{(j)}(\alpha) $ is non-zero for some $l$, some $j$, and some index $\alpha$ not contained in the block. The terms $ v_l^{(1)} \otimes \ldots \otimes v_l^{(j)}(\alpha) \otimes \cdots \otimes v_l^{(k)}$ sum to zero. However, the order $k-1$ tensors in a minimal decomposition, resulting from removing the $j$th vector from each rank one term, are linearly independent. This implies $v_l^{(j)}(\alpha) = 0$ for all $l$, a contradiction.
\end{proof}

We note that Lemma~\ref{adjoin} also appears as~\cite[Proposition 3.1.3.1]{Lan}. We use 
this lemma to relate the stabilizer of the core tensor $C$ under the congruence 
action to the set of paths with the same signature tensor $[\![C; X, X, \ldots, X]\!]$.

\begin{theorem}\label{thm:ztox}
Fix a symmetrically concise tensor $\,C$ in $(\KK^{m})^{\otimes k}$.
Let ${\rm Stab}_\KK(C)$ be its stabilizer under the congruence action
by ${\rm GL}(m,\KK)$. For any matrix $X \in \KK^{d \times m}$ of rank~$m \leq d$, we~have
\begin{equation}
\label{eq:XYZ}
 \bigl\{ Y \in \KK^{d \times m} :
[\![C; X, X, \ldots, X]\!] = [\![C; Y, Y, \ldots, Y]\!] \bigr\}
  =  \bigl\{ XZ : Z \in {\rm Stab}_\KK(C) \bigr\}. 
 \end{equation}
\end{theorem}

\begin{proof}
Suppose $[\![C; X,  \ldots, X]\!] = [\![C; Y, \ldots, Y]\!]$. 
Let $\tilde{C}$  be the $d {\times} d {\times} \cdots {\times} d$ tensor with entries
$$ \tilde{c}_{i_1 \ldots i_k} \,\,=\, \, \begin{cases} \,c_{i_1 \ldots i_k} & {\rm if} \,\,\, 1 \leq i_1,\ldots,i_k \leq m,
 \\ \quad 0 & \text{otherwise.} \end{cases}$$
Let $\tilde{X}$ be an invertible $d \times d$ matrix whose first $m$ columns are $X$, and 
likewise construct~$\tilde{Y}$. Then $[\![\tilde{C};  \tilde{X}, \ldots, \tilde{X}]\!] = 
[\![\tilde{C}; \tilde{Y}, \ldots, \tilde{Y}]\!]$. 
We multiply by $\tilde{X}^{-1}$ to get $\tilde{C} = [\![\tilde{C}; \tilde{Z}, \tilde{Z}, \ldots, \tilde{Z}]\!]$ 
where $\tilde{Z} = \tilde{X}^{-1} \tilde{Y}$ and the top-left $m \times m$ block of $\tilde{Z}$, denoted $Z$, satisfies $[\![C; Z, Z, \ldots, Z]\!] = C$.

Let $\tilde{C} = \sum_{l = 1}^r \tilde{T}_l$ be a minimal decomposition, where $\tilde{T}_l = \tilde{v}_l^{(1)} \otimes \tilde{v}_l^{(2)} \otimes \cdots \otimes \tilde{v}_l^{(k)}$ with $\tilde{v}_l^{(j)} \in \KK^m \times \{0\}^{d-m} \subseteq \KK^d$.
  We obtain another minimal decomposition of $\tilde{C}$, by acting with $\tilde{Z}$, with rank one terms $[\![\tilde{T}_l; \tilde{Z}, \tilde{Z}, \ldots, \tilde{Z} ]\!] = (\tilde{Z} \tilde{v_l}^{(1)}) \otimes (\tilde{Z} \tilde{v_l}^{(2)})   \otimes \cdots \otimes (\tilde{Z} \tilde{v_l}^{(k)})$. 
   By Lemma~\ref{adjoin}, all minimal decompositions of $\tilde{C}$ come from 
 those of $C$ by adjoining zeros. 
This means that the $d-m$ row vectors in the $(d-m) \times m$ lower-left block of $\tilde{Z}$ have 
dot product zero with every vector appearing in a minimal decomposition of $C$.
Since $C$ is symmetrically concise,
 these row vectors must be zero.
The identity $\tilde{Y} = \tilde{X} \tilde{Z}$ now implies  $Y = XZ$. This concludes the proof.  
\end{proof}

\begin{corollary} \label{cor:algrat}
Let $C \in (\KK^m)^{\otimes k}$ be a  symmetrically concise tensor whose stabilizer under congruence by $ {\rm GL}(m,\KK)$ has cardinality $n$. Then, for any matrix $X \in \KK^{d \times m}$ of rank $m$, there are $n$ matrices in $\KK^{d \times m}$ with $k$th  signature $[\![C; X, X,$ $\ldots, X]\!]$.
In particular, if the stabilizer of $C$ under congruence by ${\rm GL}(m,\KK)$ is finite then rank $m$ matrices $X$ are algebraically identifiable from the signature tensor $[\![C; X, X, \ldots, X]\!]$.
\end{corollary}

\begin{proof}
Let $Y$ be a matrix in $ \KK^{d \times m}$ 
 that satisfies $[\![C; X, X, \ldots, X]\!] = [\![C; Y, Y, $ $\ldots, Y]\!]$. 
By Theorem~\ref{thm:ztox},
we have $Y = XZ$ where $Z$ is in the stabilizer under the congruence action. If there are $n$ choices for $Z$, 
then there are $n$ choices for $Y$. 
\end{proof}

If $C \in (\KK^m)^{\otimes k}$ has trivial stabilizer under congruence, then it is
already symmetrically concise by Proposition~\ref{prop:kers}.\
We can  thus simplify Corollary~\ref{cor:algrat} as follows. 

\begin{corollary} \label{cor:trivial}
If $C\in (\KK^m)^{\otimes k}$ has trivial stabilizer under congruence then rank $m$ matrices $X \in \KK^{d \times m}$ are identifiable from the signature tensor $[\![C; X, X, \ldots, X]\!]$.
\end{corollary}

In the rest of this section we assume that $k=3$. The following example illustrates why 
Theorem~\ref{thm:ztox} and Corollary~\ref{cor:algrat} fail when $C$ is not symmetrically concise.

\begin{example} \rm \label{ex:XYZ2}
Fix the $2 {\times} 2 {\times} 2$ tensor $C = e_1 \otimes e_1 \otimes e_1$.
 Its stabilizer in ${\rm GL}(2,\RR)$~is 
$$ Z \,\,= \,\, \begin{bmatrix} 1 & * \\ 0 & * \end{bmatrix} ,$$
 where the $*$ entries can take any value in $\RR$. Setting $m=2,d=3$, we also introduce 
$$ X \,\,= \,\, \begin{small} \begin{bmatrix} 1 & 0 \\ 0 & 1 \\ 0 & 0 \end{bmatrix} \end{small} 
\qquad {\rm and} \qquad
Y \,\,= \,\, \begin{small} \begin{bmatrix} 1 & * \\ 0 & *  \\ 0 & * \end{bmatrix} \end{small}. $$  
The left hand side of \eqref{eq:XYZ} is the set of all matrices of the form $Y$.
This set strictly contains the right hand side of \eqref{eq:XYZ}, because not all matrices $Y$ are expressible as $XZ$ for some $Z$. 
  This happens because the last row of $Y$ has dot product zero with all vectors in the minimal decomposition of $C$, without being zero itself, i.e.~the tensor $C$ is not symmetrically concise.
\end{example}

Identifiability for tensors is usually studied in the context of minimal decompositions, see e.g.~\cite{Krus}. 
The following result gives conditions under which algebraic identifiability of a minimal 
decomposition implies algebraic identifiability under congruence.
We consider two decompositions of a tensor to be the same if they differ by a re-ordering of the rank one terms. 

\begin{theorem} \label{alg}
Let $\psi$ be a dictionary that is not a loop. Suppose that its core tensor
 $C = C_{\psi} \in (\RR^{m})^{\otimes 3}$ is symmetrically concise,
  has ${\rm rank}(C) = r$, and
 the number~$\delta$
of minimal decompositions of $C$ is finite.
Given a generic matrix $X \in \RR^{d \times m}$ with $m \leq d$,
there are at most $\delta \cdot \frac{r!}{(r-m)!}$
matrices $\,Y \in \RR^{d \times m}$ that have the same third order signature tensor as $X$.
\end{theorem}

\begin{proof}
We determine the number of solutions $Y$ to the tensor equation
\begin{equation}
\label{eq:wearecounting}
  \sigma^{(3)}(X) \,=\,[\![C; X, X, X]\!] \,\,\,\, = \,\,\,\, [\![C;Y,Y,Y]\!] \, = \, \sigma^{(3)}(Y). 
\end{equation}
Let $C = \sum_{l=1}^r T_l$ be a minimal decomposition. Consider a change of basis 
of $\mathbb{R}^m$ such that all standard basis vectors $e_1,\ldots,e_m$ occur  in 
the minimal decomposition. 
This exists because $C$ is symmetrically concise.
 Let $W$ be the change of basis matrix.  By Theorem \ref{thm:ztox},
it suffices to count $m \times m$ matrices $Z$ which stabilize $C' = [\![C; W, W, W]\!]$. 
This is the third signature  of the path $W\psi$, and it also has $\delta$ minimal decompositions.

Let $\,C' = \sum_{l=1}^r S_l\,$ be one of the minimal decompositions.
We have at most $r$ choices for the image of $e_1$ (up to scale) in this decomposition. Then, we have at most $r-1$ choices for $e_2$ up to scale, etc. This gives at most $\frac{r!}{(r-m)!}$ choices 
of $m \times m$ matrices $N$ with $Z = N \Lambda$, where $\Lambda$ is diagonal and invertible. 
Since $\psi$ is not a loop, the first order signature 
$ v = \psi(1)-\psi(0)$ is recoverable from the diagonal entries of the third order signature and 
$v \not= 0$ is also fixed by $Z$. Hence $Zv = v$, so $\Lambda v = N^{-1} v$. Evaluating the right hand side allows us to find $\Lambda$. 
\end{proof}

The following example attains the bound in Theorem~\ref{alg} non-trivially.

\begin{example}[$m=d=2$] \rm
Fix the dictionary $\psi = (\psi_1,\psi_2)$ with basis functions
$\,\psi_1(t) =t - 10t^2  + 10t^3\,$ and $\,\psi_2(t) = 11 t - 20 t^2 + 10 t^3$.
By \eqref{eq:tripleintegral}, the core tensor equals
\[  C_\psi \,\, = \,\, \,\frac{1}{42}
\left[
\begin{array}{cc|cc}
\phantom{-}7 & -8 & \phantom{-}37 & -8 \\
-8 & \phantom{-}37 & -8 & \phantom{-} 7
 \end{array}
\right] .
\]
Using  {Macaulay2} \cite{M2}, we find that this tensor is symmetrically concise and has rank two, and a unique rank two decomposition. The stabilizer consists of two matrices:
\[ \begin{bmatrix} 1 & 0 \\ 0 & 1 \end{bmatrix} ,  \begin{bmatrix} 0 & 1 \\ 1 & 0 \end{bmatrix} . \] 
Our upper bound of $\delta \cdot \frac{r!}{(r-m)!} = 1 \cdot 2 = 2$ on the size of the stabilizer is attained.
The stabilizer shows that $C_\psi$ is unchanged under swapping the coordinates $\psi_1$ and $\psi_2$.
\end{example}

\section{Identifiability for Generic Dictionaries}
\label{sec4}

Our main motivation for 
working with
third order signatures is that, under reasonable
hypotheses on the dictionary $\psi$ and the 
ambient dimension $d$, the path $X \psi$ 
can be recovered uniquely from $\sigma^{(3)}(X)$. We begin by placing this in the context of lower order signature tensors via a description of the set of paths with the same first and second order signatures. Thereafter, we turn to the universal variety 
$\mathcal{U}_{m,3}$ of all third order signature tensors, and we apply results
from invariant theory to show that generic dictionaries are identifiable.

We fix a dictionary $\psi$. The paths 
represented by this dictionary are paths $X \psi$ as $X$ varies over $d \times m$ matrices. 
In this $md$-dimensional space, there is an $m(d-1)$-dimensional linear space of paths 
with the same first signature $\sigma^{(1)}(X) = X\psi(1)-X\psi(0)$. 
For second order signatures, the non-uniqueness of path recovery is quantified by the stabilizer of the
$m \times m$ core matrix $C= C_\psi$ under the matrix congruence action:
$$ {\rm Stab}_\RR(C)\,\,  = \,\, \bigl\{\, X \in {\rm GL}(m,\RR) \,\,:\,\, X C X\T = C\, \bigr\}. $$

\begin{proposition} \label{prop:viereins}
The stabilizer   of a generic  $m {\times} m$ core matrix $C$  is a variety of dimension
$\binom{m}{2}$. Setting ${\rm det}(X) > 0$,
it is conjugate to the symplectic group ${\rm Sp}(m,\RR)$ intersected
with the codimension $m$ group of matrices that fix a given vector in $\RR^m$.
\end{proposition}

\begin{proof}
The symplectic group ${\rm Sp}(m,\RR)$ is the set of all endomorphisms
of $\RR^m$ that fix a skew-symmetric bilinear form of maximal rank.
If $m$ is even then this is one of the classical semi-simple Lie groups.
If $m$ is odd then ${\rm Sp}(m,\RR)$ can be realized by
extending the matrices in ${\rm Sp}(m-1,\RR)$ by a column of arbitrary entries.
In both cases, we have ${\rm dim}({\rm Sp}(m,\RR)) = \binom{m+1}{2}$.
The core matrix equals $C = u u\T + Q$ where
$u$ is a general column vector and $Q$ is a general
skew-symmetric matrix. Our stabilizer consists of all $m \times m$ matrices $X$
that satisfy $Xu = u$ and $X Q X\T = Q$. The second condition 
defines the symplectic group, up to change of coordinates,
and the first condition specifies a general linear space of codimension $m$.
\end{proof}

The matrix $C$ in Proposition~\ref{prop:viereins}
is a generic point in the universal variety $\mathcal{U}_{m,2}$.
We now consider an $m {\times} m {\times} m$ core tensor $C$ that is  generic 
in the universal variety $\mathcal{U}_{m,3}$. This is the third signature of a
dictionary $\psi$ which is generic in the sense of Example~\ref{ex:gendic}.
 In the following identifiability  result, the field $\KK$ can be either  $\RR$ or  $\CC$.
 
\begin{theorem} \label{thm:genericpoint}
Let $C$ be an $m {\times} m {\times} m$ tensor that is a generic point
in the  variety $ \mathcal{U}_{m,3}$. The
stabilizer of $C$ under  the congruence action by ${\rm GL}(m,\KK)$ is trivial. 
\end{theorem}

\begin{proof}
We work over the complex numbers $\CC$ and show that the
complex stabilizer of the real tensor $C$ is trivial: it consists only
of the scaled identity matrices $\eta I$, where $\eta^3=1$.
For $m \leq 3$, this result is established by a direct
Gr\"obner basis computation in {maple}.
For $m \geq 4$ we use the following parametrization of the universal variety $\mathcal{U}_{m,3}$.
Let $P$ be a generic vector in $\CC^m$, and let
$Q$ be a generic skew-symmetric $m \times m$ matrix. Following
the definitions in  \cite[\S 4.1]{AFS}, we take $L$ to be a generic element 
in the space ${\rm Lie}^{[3]}(\CC^m)$ of
homogeneous Lie polynomials of degree $3$. Then 
\begin{equation}
\label{eq:CXQL}
 C \,\,  = \,\,
\frac{1}{6} P^{\otimes 3}
\,+\, \frac{1}{2} ( P \otimes Q + Q \otimes P) \,+ \, L .
\end{equation}
Indeed, $P + Q + L$ is a general Lie polynomial of degree $\leq 3$,
and the right hand side in  \eqref{eq:CXQL} is the degree $3$ component
in the expansion of its exponential, see \cite[Example 5.15]{AFS}.
The constituents $P$, $Q$ and $L$ are recovered from $C$
by taking the logarithm of $C$ in the tensor algebra and extracting the 
homogeneous components of degree $1$, $2$ and $3$.
In particular, since these computations are equivariant
with respect to the congruence action by ${\rm GL}(m,\KK)$,
 the stabilizer of $C$ is contained in the stabilizer of $L$.

By \cite[Proposition 4.7]{AFS}, a basis for
the vector space ${\rm Lie}^{[3]}(\CC^m)$ consists of the
bracketings of all Lyndon words of length three on
the alphabet $\{1,2,\ldots,m\}$. The number of these Lyndon triples is
$\frac{1}{3}(m^3-m)$. The group $G = {\rm GL}(m,\CC)$ acts irreducibly on 
 ${\rm Lie}^{[3]}(\CC^m)$. By comparing dimensions, we see that
 \begin{equation}
 \label{eq:LieS21}
  {\rm Lie}^{[3]}(\CC^m) \,\, \simeq\,\, S_{(2,1)}( \CC^m) . 
  \end{equation}
  The right hand side is the irreducible  $G$-module
  associated with the partition $(2,1)$ of the integer $3$;
  see~\cite{Lan}. By the Hook Length Formula, the  vector space dimension
  of (\ref{eq:LieS21}) equals
$\frac{1}{3}(m^3-m)$. This number
exceeds the dimension $m^2$ of
the group, since $m \geq 4$. 
The map $C \mapsto L$ from the universal variety 
$\mathcal{U}_{m,3}$ to the $G$-module in \eqref{eq:LieS21} is surjective,
since the homogeneous Lie polynomial $L$ in \eqref{eq:CXQL} can be
chosen arbitrarily.

We now apply Popov's classification \cite{AP, P} of irreducible $G$-modules with non-trivial
generic stabilizer. A recent extension to arbitrary fields due to
 Garibaldi and Guralnick can be found in \cite{GG}. A very special case of these general results
says that the stabilizer 
of a generic point $L$ in the $G$-module $S_{(2,1)}(\CC^m)$ is trivial. This implies that the
stabilizer of the core tensor $C$ under the congruence action of $G$ is trivial.
\end{proof}

We conclude from Theorem \ref{thm:genericpoint} and Corollary~\ref{cor:trivial}
that the paths which are representable in a generic dictionary
are  identifiable from their third order signature.

\begin{corollary} \label{cor:genericid}
Let $ m \leq d$ and let $C \in \mathcal{U}_{m,3}$ be a generic dictionary.
Given $X \in \RR^{d \times m}$ of rank $m$, the only real
solution to $[\![C; X, X, X]\!]  = [\![C;Y,Y,Y]\!]$ is $\,Y= X$.
\end{corollary}

We close this section with a remark about
equations defining $\mathcal{U}_{m,3}$.
The entries $p_k$ of the vector $P$
and the entries $q_{ij}$ of the skew-symmetric matrix $Q$
are recovered from the entries $c_{ijk}$ of the core tensor $C$ 
by the identities
\begin{equation}
\label{eq:xqc}
\begin{matrix}
p_k q_{ij} & = & 
\frac{1}{2}( c_{kij} + c_{ikj} + c_{ijk})\, -\,\frac{1}{2} (c_{kji} + c_{jki} + c_{jik} ) ,\\
 p_i p_j p_k &=& \,\,\,\, c_{kij} + c_{ikj} + c_{ijk}\,\,\,+\,\quad c_{kji} + c_{jki} + c_{jik}  .
 \end{matrix}
 \end{equation}
 These formulas are a very special case of \cite[Proposition 18]{Gal}.
 Here $(i,j,k)$ runs over triples in $\{1,2,\ldots,m\}$.
The linear forms on the right hand side \eqref{eq:xqc}
span the shuffle linear forms, i.e.~linear
equations that cut out \eqref{eq:LieS21} as a 
subspace of $(\CC^m)^{\otimes 3}$.

We regard $P$ as a column vector of length $m$.
For any matrix $A$ we write ${\rm vec}(A)$ for  the
row vector whose coordinates are the entries of $A$.
We apply this to the
symmetric matrix $P P\T$ and the skew-symmetric matrix $Q$,
and we concatenate the resulting row vectors.
Then the following matrix has $m$ rows
and $2m^2$ columns:
\begin{equation*}
\label{eq:hankel}
H \quad = \quad P \cdot \bigl[ \,{\rm vec}(P P\T) \,\,\, {\rm vec}(Q) \,\bigr] .
\end{equation*}
By construction, the matrix $H$ has rank $1$, and we obtain
$P$ and $Q$ from its rank $1$ factorization.
Each entry of $H$ is one of the monomials on the left hand side in \eqref{eq:xqc}.

Let $H[C]$ denote the $m  \times 2m^2 $ matrix that is obtained from $H$
by replacing each monomial by the corresponding
linear form on the right hand side of \eqref{eq:xqc}.
Thus the entries of $H[C]$ are linear forms in $C$.
We have shown that the $2 \times 2$ minors
of $H[C]$ cut out the variety $\mathcal{U}_{m,3}$.
A vast generalization of this fact is due to Galuppi~\cite{Gal}. His results also
imply that the $2 \times 2$ minors of $H[C] $ generate the prime ideal
of $\,\mathcal{U}_{m,3}$.

\section{Piecewise Linear Paths are Identifiable}
\label{sec5}

In this section we prove that piecewise linear paths in real
$d$-space with $m \leq d$ steps are uniquely recoverable from their third order signatures.
As before, we take $\KK $ to be $ \RR $ or $ \CC$.
Let $C_{\rm axis} \in \KK^{m \times m \times m}$
 be the piecewise linear core tensor in \eqref{axis} and $S$ any tensor in the orbit
  \begin{equation*}
 \bigl\{\, [\![ C_{\rm axis}; X, X, X ]\!] \in \KK^{d \times d \times d}\, :\, X \in \KK^{d \times m} \,\bigr\}.
\end{equation*}
We show that there is a unique matrix $X$, up to 
 third root of unity, such that $S = [\![ C_{\rm axis}; X, X, X ]\!]$. 
This proves Conjectures 6.10 and 6.12 in \cite{AFS} for $m \leq d$. In particular,
if the field $\KK$ is the real numbers $\RR$,
the matrix $X$ can be uniquely recovered from $S$.

\begin{lemma}\label{lem:applyd}
Let $X \in \KK^{m \times m}$ be in the stabilizer of $\, C_{\rm axis}$
under congruence. If $\,e_m = (0,\ldots,0,1)\T$ is an eigenvector of $X$
 then $e_m$ is also an eigenvector of $X\T$.
\end{lemma}

\begin{proof}
Any matrix $X$ in the stabilizer of $ C = C_{\rm axis}  $ also stabilizes the first and second order signatures, 
up to scaling by third root of unity, by~\eqref{eq:shufflerel}.
 The core tensor  $C$ represents a path from $(0,\ldots,0)$ to $(1,\ldots,1)$, 
 so the first order signature is $b = (1,\ldots,1)\T$. The matrix $X$ satisfies 
 $Xb = \eta b$, where $\eta^3 = 1$,
 so $b$ is an eigenvector of $X$.
 By \cite[Example 2.1]{AFS}, the signature matrix of the piecewise linear dictionary~is
 $$ C_2 \,\,\,=\,\,\, \begin{bmatrix}
\frac{1}{2} & 1 & \cdots & 1 \\
0 & \ddots & \ddots & \vdots \\
\vdots & \ddots & \frac{1}{2} & 1\\
0 & \ldots & 0 & \frac{1}{2}
\end{bmatrix}.
$$
Since $C_2$ differs from $XC_2 X\T$ by a third root of unity, denoted $\eta'$, we have
\begin{equation}
\label{eq:transeig}
X C_2 X\T \,=\, \eta' C_2 \quad
\Longrightarrow \quad \eta' C_2^{-1} \,=\, X\T C_2^{-1} X \quad
\Longrightarrow \quad \eta' C_2^{-1}b\, =\, \eta X\T C_2^{-1} b.
\end{equation}
Hence $\,v := \frac{1}{2}C_2^{-1}b  = (\pm 1, \mp 1, \ldots, -1, 1)\T\,$
is an eigenvector of $X\T$.

Consider the matrix obtained from $C$ by multiplying by $v$ along the second index, $D = [\![ C; \cdot , v , \cdot ]\!]$ or $D_{ik} = \sum_{j} c_{ijk} v_j$. The matrix $D$ is diagonal, by the following direct computation. If $i < k$, we get
\begin{equation*}
D_{ik}  \,\,=\,\,  \sum_{j = 1}^m (-1)^{m+j}
c_{i j k}  =  (-1)^m \biggl( \frac{1}{2}(-1)^i + \sum_{i < j < k}
(-1)^j +\frac{1}{2}(-1)^k \biggr) \,\, =\,\,   0.
\end{equation*}
If $i > k$ all entries $c_{ijk}$ vanish hence the sum is zero.
If $i = k$, the only non-zero entry of $C$ that appears in the sum is $c_{iii}$ and we obtain $D_{ii} = \frac{1}{6}(-1)^{m+i}$.
Also, by definition of $D$, we find that the matrix $X$ is in its stabilizer under congruence, up to scaling by third root of unity,
$D = \frac{\eta'}{\eta} XDX\T $.

Suppose that $e_m$ is an eigenvector of $X$. 
 By the same argument as in~\eqref{eq:transeig} we find that 
 $D^{-1} e_m = 6 e_m$ is an eigenvector of $X\T$.
Hence, $e_m$ is an eigenvector of $X\T$.
\end{proof}

\begin{theorem} \label{thm:sixtwo}
The stabilizer of the piecewise linear core tensor $C = C_{\rm axis}$ 
under the congruence action $C \mapsto [\![C; X, X, X]\!]$ by matrices $X \in {\rm GL}(m,\KK)$ is trivial.
\end{theorem}

\begin{proof}
Let $X $ be in ${\rm Stab}_\KK(C)$.
  Evaluating $C = [\![C; X, X, X]\!]$ at coordinate $(i,j,k)$ implies $c_{ijk}$ equals
  \begin{small}
$$  \sum_{1 \leq \alpha \leq m} \frac16 x_{i \alpha} x_{j \alpha} x_{k \alpha}\, \,\,+\!\! 
 \sum_{1 \leq \alpha < \beta \leq m} \frac12 x_{i \alpha} x_{j \alpha} x_{k \beta} 
\, \,\,+\!\! \sum_{1 \leq \alpha < \beta \leq m} \frac12 x_{i \alpha} x_{j \beta} x_{k \beta} 
\, \,\,+\!\!  \sum_{1 \leq \alpha < \beta < \gamma \leq m}\!\! x_{i \alpha} x_{j \beta} x_{k \gamma} .
 $$
 \end{small}
Here the constants in (\ref{axis}) were substituted for the  entries $c_{\alpha \beta \gamma}$ of $C$.
We can express this equation as the dot product $c_{ijk} =  f_{jk} \cdot X_i\T  = \sum_{\alpha = 1}^m f_{jk}(\alpha) X_i(\alpha)$, where $X_i$ is the $i$th row of $X$ and
 $f_{jk}$ denotes the row vector with $\alpha$-coordinate 
$$ f_{jk} (\alpha) \,\,\,= \,\,\,\frac16 x_{j \alpha} x_{k \alpha}
\, +\, \sum_{\beta > \alpha} \frac12 x_{j \alpha} x_{k \beta} 
\,+\, \sum_{\beta > \alpha} \frac12 x_{j \beta} x_{k \beta}
\, + \,\sum_{\gamma > \beta > \alpha} x_{j \beta} x_{k \gamma} .$$
When $j > k$, the entry $c_{ijk}$ vanishes, for all $1 \leq i \leq m$. Hence the vector $f_{jk}$ 
for $j > k$ has dot product zero with all rows of $X$. Since the rows of $X$ span $\KK^m$, 
we see that $f_{jk}$ is the zero vector, and  the last entry 
$f_{jk}(m) =\frac16 x_{jm} x_{km}$ vanishes for all $j \neq k$. 

We can express the entries $c_{ijk}$ as a different dot product.
Namely, factoring out the terms involving the $j$th row, we obtain
$c_{ijk} = g_{ik} \cdot X_j\T$,
where $g_{ik}$ is the row vector with $\beta$-coordinate 
$$ g_{ik}(\beta) \,\,\,=\,\,\, \frac16 x_{i \beta} x_{k \beta} 
\,+\, \sum_{\gamma > \beta} \frac12 x_{i \beta} x_{k \gamma} 
\,+\, \sum_{\alpha < \beta} \frac12 x_{i \alpha} x_{k \beta} 
\,+\, \sum_{\gamma > \beta > \alpha} x_{i \alpha} x_{k \gamma}.$$ 
For all $i > k$, the entry $c_{ijk}$ vanishes.
This means that the dot product of $g_{ik}$ with all rows of $X$ is zero, hence $g_{ik}$ is the zero vector. Its $m$th entry $g_{ik}(m)$ equals
$$ \frac16 x_{i m} x_{k m} \,+\, \sum_{\alpha = 1}^{m-1} \frac12 x_{i \alpha} x_{k m} 
\,\,\,=\,\,\, \frac{x_{km}}{2} \left( \sum_{\alpha = 1}^m x_{i \alpha} - \frac23 x_{im} \right) .$$
Since $X$ stabilizes the first order signature, up to scaling by third root of unity $\eta$, the rows of $X$ sum to $\eta$, hence $g_{ik}(m) = \frac{\eta}2 x_{km} - \frac13 x_{km} x_{im}$ for all $i > k$. By the previous paragraph, the second term vanishes and, setting $i = m$, we deduce that $x_{km} = 0$ for all $1 \leq k < m$. This implies that the last column of $X$ is parallel to the $m$th standard basis vector $e_m$, and hence that $e_m$ is an eigenvector of~$X$. 

By Lemma~\ref{lem:applyd}, $e_m$ is also an eigenvector of $X\T$. Thus, all entries in
the last row of $X$ vanish except the last. This means that $X$ has the block diagonal structure
$$ X \,\,=\,\, \begin{small} \begin{bmatrix} * & \cdots & * & 0 \\ \vdots & & \vdots & \vdots \\ 
* & \cdots & * & 0 \\  0 & \cdots & 0 & 1 \end{bmatrix} \end{small} .$$
The $*$ entries represent unknowns in an $(m-1) \times (m-1)$ block which we call $X'$. 

We now observe that $X'$ stabilizes $C'$, the axis core tensor
in $\KK^{(m-1) \times (m-1) \times (m-1)}$ arising
from $C$ by restricting to indices $1 \leq i,j,k \leq m{-}1$. From
$C = [\![C; X, X, X]\!]$ we have 
$ c_{ijk} = \sum_{\alpha, \beta, \gamma = 1}^m c_{\alpha \beta \gamma}
 x_{i \alpha} x_{j \beta} x_{k \gamma} $. Since
$x_{u v } = 0$ whenever $u < m = v$, this simplifies to
$$ c_{ijk} \,\,= \sum_{\alpha, \beta, \gamma = 1}^{m-1} c_{\alpha \beta \gamma} x_{i \alpha} x_{j \beta} x_{k \gamma} \quad \text{for} \quad 1 \leq i,j,k \leq m-1 .$$
Hence $X'$ is in the stabilizer of $C'$. 
The proof of Theorem \ref{thm:sixtwo} is concluded by induction on $m$,
given that the assertion can be tested for small $m$ by a direct computation.
\end{proof}

We deduce from Theorem~\ref{thm:sixtwo} and Corollary~\ref{cor:trivial} 
that piecewise linear paths in~$\RR^d$, consisting of at most~$d$ steps, can be uniquely recovered from the third order signature. This uses the fact that the
upper triangular
tensor $C_{\rm axis}$ is symmetrically concise.

\begin{corollary} \label{cor:plid}
Let $ m \leq d$, let $C = C_{\rm axis}$, and  $X \in \RR^{d \times m}$ of rank $m$. The only real
solution to the tensor equation $[\![C ; X, X, X]\!]  = [\![C ;Y,Y,Y]\!]$ is 
the matrix $Y= X$.
\end{corollary}

\section{Numerical Identifiability}
\label{sec7}

A path in $\RR^d$, coming from a dictionary of size~$m$, can be recovered from its signature tensor $S \in \RR^{d \times d \times d}$ by solving a system of $d^3$ equations in $md$ unknowns (see Section~\ref{sec2}). This can be done in principal using Gr\"obner basis methods. However, such methods are infeasible when  $m$ and $d$ are large,
or when the data $S$~is inexact or noisy. In such cases we instead minimize the distance between $S$ and the
set of tensors $[\![C; X,X,X ]\!]$ as $X$ ranges over  $ \RR^{d \times m}$.
We seek the global minimum of the cost function
\begin{equation}
\label{eq:costf}
g : \mathbb R^{d \times m} \rightarrow \mathbb R \, ,\,\,\,\
X \, \mapsto \, \bigl\| \,[\![ C ; X, X, X ]\!] \,-\, S \,\bigr\|^2,
\end{equation}
where $\| \cdot \|$ denotes the Euclidean norm in tensor space, the {\em Frobenius norm}.
We first comment on the algebraic complexity of this cost function. Then, we quantify the
numerical identifiability of recovering paths from third order signatures. 

The number of critical points of~\eqref{eq:costf} is the {\em ED degree} \cite{DHOST} of the orbit of an $m {\times} m {\times} m$ 
 tensor $C$ under multiplication by $d \times m$ matrices.
 For $m=d$, it is the ED degree of an orbit of 
 the congruence action of ${\rm GL}(m,\CC)$ on the space
 $ \CC^{m \times m \times m}$.

\begin{example}[$d=m=2$] \rm \label{ex:45}
Fix the core tensor $C_{\rm axis}$ of format $2 \times 2  \times 2$.
Its orbit under the congruence action
is the degree $6$ variety $\mathcal{L}_{2,3,2}$ in~\cite[Table 3]{AFS}, defined by $9$ quadrics.
A computation reveals that the ED degree of $\mathcal{L}_{2,3,2}$ is $15$.
For generic data $S \in \RR^{2 \times 2 \times 2}$, our minimization problem
 has $45 = 15 \times 3$ critical points $X$ in $\CC^{2 \times 2}$.
 Each critical point on $\mathcal{L}_{2,3,2}$ corresponds to a 
 triple of matrices $\eta X$, with $ \eta^3=1$.
\end{example}

The ED degree specifies the algebraic degree of the coordinates of the 
optimal solution to~\eqref{eq:costf}, given rational data $S$.
 This degree can drop for special tensors~$S$.

\begin{example}[The skyline path] \label{ex:skyline} \rm
This path has $13$ steps in $\RR^2$, the columns of
\setcounter{MaxMatrixCols}{15}
$$ Y \,\, = \,\,
\begin{bmatrix}
 \,1 &  0 &  1 &   \phantom{-}0  &   1  &   0  &   1  & \phantom{-} 0  &   1  &   0   &  1  &  \phantom{-} 0  &  1 \, \\
\, 0 &  1 &  0 &                    -1  &   0  &   2  &   0  &                    -2  &   0 &    1  &   0  &                     -1 &  0 \, 
\end{bmatrix}.
$$
  Its third order signature tensor is $[\![ C_{\rm axis} ; Y, Y, Y ]\!]$, where $C_{\rm axis}$ has format $13 \times 13 \times 13$:
\begin{equation*}
S_{\rm skyline} \,\,=\,\, [\![ C_{\rm axis} ; Y, Y, Y ]\!] \,\,\,=\, \,\, \,\frac{1}{6}
\left[
\begin{array}{cc|cc}
343 & 0 \, & \, -84 & 18 \\
84 & 18 \,& \, -36 & 0 \\
 \end{array}
\right] .
\end{equation*}
Using Gr\"obner bases, we compute the best approximation of $S_{\rm skyline}$
 by a piecewise linear path with $m=2$ steps. The solution is 
 the path  given by the two columns of
$$ X^* \,\, = \,\,
\begin{bmatrix}
\,\, a & \phantom{-} a \,\, \\
\,\, b & -b \,\,
\end{bmatrix}
$$
where
$ a = 3.4952680660622583405....$ and $ b = 1.218447154323916453....$.
These are algebraic numbers of degree $12$.
The Euclidean distance between the tensors is
\begin{equation}
\label{eq:336}
 \|\, [\![ C_{\rm axis} ; X^*, X^*, X^* ]\!] \,-\, S_{\rm skyline}\, \|\, \,\, \approx \,\,\, 3.36 .
\end{equation}
\end{example}

In the rest of this section, we address the numerical identifiability of recovering a path from its signature tensor.
We fix a core tensor $C$ with trivial stabilizer under congruence action. Consider the set $\mathcal{N}(C,d)$
of signature tensors $S$ for which $\{ Y \in \RR^{d \times m} : S = [\![C; Y, Y, Y]\!] \}$ 
has cardinality 
at least $2$.
These are instances for which path recovery is non-identifiable. 
However, even if a path is identifiable from its signature tensor in the exact sense of Sections~\ref{sec3}--\ref{sec5},
different paths may lead to numerically indistinguishable signatures.
We quantify this via the distance to the set $\mathcal{N}(C,d)$ of {\em ill-posed instances}. We define the {\em numerical non-identifiability} of a pair $(C,X)$ to be
$$ \kappa(X,C) \quad = \quad
\frac{\| X \|^3 \cdot \| C \| }{{\rm inf}_{S \in \mathcal{N}(C,d)} \| [\![C; X, X, X]\!] - S \| },$$
where $\| \cdot \|$ denotes the Frobenius norm.
When the numerical non-identifiability is large, this reflects the proximity of $[\![C; X, X, X]\!]$ to a non-identifiable tensor. Conversely, a small value of the numerical non-identifiability means that all close-by tensors $[\![C; X, X, X]\!]$ are also identifiable. 
We give upper and lower bounds on $\kappa(X,C)$, in terms of the  flattenings of~$C$ and the condition number of the rectangular matrix~$X$, $\kappa(X) = \| X \| \cdot \| X^{+} \|$ where $X^+$ denotes the pseudo-inverse.

We first remark on connections between the numerical non-identifiability and the condition number. Following~\cite[Section O.2]{BC},
and setting $S = [\![ C ; Y, Y, Y ]\!]$,
 the condition number of our recovery problem is 
\begin{equation}
\label{eq:conddef}
{\rm cond}(X,C) \,\,\, = \,\, \lim_{\delta \to 0} \sup_{\| [\![C; X, X, X]\!] - S \| \leq \delta} \frac{\| X - Y \|}{\| [\![C; X, X, X]\!] - S ]\!] \|} \cdot \frac{\| [\![C; X, X, X]\!] \|}{\| X \|}.
\end{equation}
 The condition number records how much the recovered matrix can change with small changes to the signature tensor. 
When the condition number is finite, the matrix can be recovered uniquely using symbolic computations.
However, when the condition number is large, small changes in the signature 
induce large changes in the recovered matrix, a problem for numerical computations.
Following the approach introduced by~\cite{Ren} in the context of linear programming, the condition number is often determined by the inverse distance to the set of instances with infinite condition number~\cite{BC,Demmel}.
On the set of ill-posed instances $\mathcal{N}(C,d)$ the condition number is infinite, because the problem is non-identifiable. Hence the numerical non-identifiability gives a lower bound on the inverse distance to the instances with infinite condition number. 
Condition numbers for algebraic identifiability can be defined using the local set-up described in~\cite{BV}.
Proving a {\em condition number theorem} to relate~\eqref{eq:conddef} to our inverse distance would be an interesting topic for further study. 



\begin{theorem} \label{thm:kxc}
The numerical non-identifiability of the pair $(X,C)$, for a matrix $X \in \RR^{d \times m}$ and a tensor $C \in \RR^{m \times m \times m}$ with trivial stabilizer, satisfies the upper bound
\begin{equation} \label{eq:ineqlaura}
\kappa(X,C) \quad \leq \quad  \kappa(X)^3 \left(  \frac{ \| C \| }{ \max(\varsigma_m^{(1)},\varsigma_m^{(2)},\varsigma_m^{(3)})} \right) ,
\end{equation}
where $\varsigma_m^{(i)}$ denotes the smallest singular value of the $i$th flattening of the tensor $C$. 
\end{theorem}

\begin{proof}
We aim to bound $\kappa(X,C)^{-1}$, the distance of $[\![C; X, X, X]\!]$ to the locus of non-identifiable tensors, from below. Since $C$ has trivial stabilizer, Corollary~\ref{cor:trivial} implies that all non-identifiable tensors must be of the form $[\![C; Y, Y, Y]\!]$ where $Y \in \RR^{d \times m}$ has rank strictly less than $m$. The flattenings of such tensors have rank strictly less than $m$, so it suffices to lower-bound the distance of the flattenings to the much larger set $\{A \in \RR^{d \times d^2} : \rank(A) < m\}$. We have
\begin{align*} \| X \|^3 \cdot \| X^{+} \|^3 \cdot \| C \| \cdot \kappa(X,C)^{-1} \,
\, & \geq \min_{{\rm rank}(A) < m } \| X C^{(i)} (X \otimes X)\T    - A \| \cdot \| X^{+} \|^3
\\
& = \min_{{\rm rank}(A') < m }  \| X (C^{(i)} -A') (X \otimes X)\T   \| \cdot \| X^{+} \|^3
\\ 
& \geq \min_{{\rm rank}(A') < m } \| C^{(i)} - A' \| 
\\ 
& = \quad \varsigma_m^{(i)},
\end{align*}
where $X \otimes X \in \RR^{d^2 \times m^2}$ is the 
Kronecker product of the matrix $X$ with itself, $A' \in \RR^{m \times m^2}$, and $\| \cdot \|$ 
is~the Frobenius norm. The chain of inequalities 
holds for $i = 1,2,3$, and the claim follows. 
\end{proof}

We quantify the suitability of a core tensor $C$, with trivial stabilizer under congruence, for path recovery. 
We define the numerical non-identifiability of $C$ to be the 
smallest number  $\kappa(C)$ satisfying
\begin{equation*}
\label{eq:condC}
\kappa(C)  \quad \geq \quad
\frac{\kappa(X,C)}{ \kappa(X)^3 }
\end{equation*}
for all $X \in \RR^{d \times m}$ of rank $m$ and all $d$. 
From Theorem~\ref{thm:kxc}, we obtain the following.

\begin{corollary} \label{cor:upperbound}
The numerical non-identifiability of $C \in \RR^{m \times m \times m}$, with trivial stabilizer under congruence, satisfies 
$$
\kappa(C) \quad  \leq \quad
\frac{ \| C \| }{ \max(\varsigma_m^{(1)},\varsigma_m^{(2)},\varsigma_m^{(3)})}.
$$
\end{corollary}

\begin{proof}
Divide (\ref{eq:ineqlaura}) by $\kappa(X)^3$.
The supremum of the left hand side, as $X$ ranges over $\RR^{d \times m}$ for all $d$, is equal to
$\kappa(C)$. Hence $\kappa(C)$ is bounded by the right hand side.
\end{proof}

We now bound the numerical non-identifiability of the piecewise linear dictionary.

\begin{corollary} \label{cor:condaxis}
The numerical non-identifiability of $C_{\rm axis}$ is at most $6 \| C_{\rm axis} \|$.
\end{corollary}

\begin{proof}
We show that the singular values of the second flattening $C^{(2)} \in \RR^{m \times m^2}$ are at least~$\frac16$. 
The $j \times (i,k)$ entry is $c_{ijk}$. Since the entries of $C$ are zero unless $i \leq j \leq k$, the flattening has an $m \times m$ block, indexed by $j \times (i,i)$, which equals $\frac{1}{6}$ times the identity matrix $I$. Let $B$ denote the $m \times ( m^2 - m)$ matrix obtained by removing these $m$ columns.  Then $C^{(2)} (C^{(2)})\T=
\frac{1}{36} I + B B\T$. The singular values of $C^{(2)}$are the square roots of the eigenvalues of $C^{(2)} (C^{(2)})\T$. Consider an eigenvector $v$ of $B B\T$ with eigenvalue~$\lambda$. Then $\lambda \geq 0$ because $B B\T$ is positive semi-definite and $v$ is an eigenvector of $C^{(2)} (C^{(2)})\T$ with eigenvalue $\frac1{36} + \lambda$. Hence the singular values of $C^{(2)}$ are bounded from below by~$\frac16$. 
\end{proof}

Proposition~\ref{prop:kers} shows that the recovery problem is ill-posed if the tensor is not symmetrically concise.
What follows is a numerical analogue to Proposition~\ref{prop:kers}.

\begin{proposition} \label{prop:lowerbound}
Let $C \in \RR^{m \times m \times m}$ and $C^{({\rm all})}$ be the $m {\times} 3m^2$ matrix obtained by concatenating the three flattening matrices $C^{(i)}$. If $\varsigma_m$ is the smallest singular value
of $C^{({\rm all})}$ then
$$\kappa(C) \quad \geq \quad \frac{\| C \| }{7 m^{3/2} \varsigma_m }.$$
\end{proposition}

\begin{proof}
 We compute the distance to a tensor $S$ in the orbit of $C$ that is not symmetrically concise.
 This gives an upper bound for     the minimal distance to the set of ill-posed instances.
Consider $S = [\![C; I - vv\T, I - v v\T, I - v v\T]\!]$, where $v$ is the left singular vector corresponding to the singular value $\varsigma_m$ of $C^{({\rm all})}$. Then 
$v$ is in the kernel of all three flattenings of $S$, 
which means that $S$ is not symmetrically concise.

We have $ v\T C^{({\rm all})} = \varsigma_m w\T$ where $w$ is the right singular vector of length $3m^2$, corresponding to singular value $\varsigma_m$. We define $w_i$ such that $w$ is the stacking of $w_1, w_2, w_3$ with each $w_i$ of length $m^2$. Then $v\T C^{(i)} = \varsigma_m w_i\T$ hence $\| v\T C^{(i)} \| = \varsigma_m \| w_i \| \leq \varsigma_m \| w \| \leq \varsigma_m$. We use this to upper bound the distance from $C$ to $S$, as follows:
\begin{align*} \| C - S \| &\,\,=\,\, \| [\![ C ; vv\T, I, I ]\!] + [\![ C ; I, vv\T, I ]\!] + [\![ C ; I , I , vv\T ]\!] - [\![ C ; vv\T, vv\T, I ]\!] \\
&\qquad - [\![ C ; I , vv\T, vv\T ]\!] - [\![ C ; vv\T, I , vv\T ]\!] + [\![ C ; vv\T, vv\T, vv\T ]\!]\| \\
&\,\,\leq \,\,\biggl( \,\sum_{i=1}^3 \| v\T C^{(i)} \| + \| v\T C^{(i)} \| \| vv\T \| \biggr) \,+\, \| v\T C^{(1)} \| \| vv\T \|^2 
\quad \leq \quad 7 \varsigma_m .
\end{align*}

We have $\kappa(C) \geq \frac{\kappa(X,C)}{\kappa(X)^3}$ for all matrices $X$ of rank $m$. In particular, $\kappa(C) \geq \frac{\kappa(I,C)}{\kappa(I)^3}$. By definition of the numerical non-identifiability
\begin{align*}
\kappa(I,C) &\,\,=\,\, \frac{\| I \|^3 \| C \|}{\inf_{\tilde S \in \mathcal N(C,m)} \| C - \tilde S \|} 
\,\,\geq \,\,\frac{ m^{3/2} \| C \|}{\| C - S \|} \,\,\geq \,\, \frac{ m^{3/2} \| C \|}{7 \varsigma_m},
\end{align*}
since $\| I \| =\sqrt{m}$. The condition number of $I$ is $m$, so the claim follows.
\end{proof}

\begin{figure}[htpb]

\definecolor{uppercol}{HTML}{3438C9}
\definecolor{lowercol}{HTML}{C93834}
\definecolor{shadecol}{HTML}{C6C6C6}

\begin{tikzpicture}
\pgfplotsset{every axis y label/.append style={yshift=-25pt}}
\begin{semilogyaxis}[
	width=.33\textwidth,
    xlabel={dimension $m$},
    ylabel={$\kappa(C_{\rm axis}$)},
    ylabel shift=-5pt,
    ytick={1e-2,1e+4,1e+10,1e+16},
    yticklabels={},
    ymajorgrids,
    xmin=2, xmax=25,
    ymin=1e-2, ymax=1e+16,
    legend pos=north west,
]
 
\addplot[name path=upper,color=uppercol,mark=none,ultra thick] table {PL_upperbound.dat};
\addplot[name path=lower,color=lowercol,mark=none, ultra thick] table {PL_lowerbound.dat};
\legend{upper bound,lower bound}
\addplot[shadecol] fill between[of=upper and lower];
\end{semilogyaxis}
\end{tikzpicture}
\begin{tikzpicture}
\pgfplotsset{every axis y label/.append style={yshift=-25pt}}
\begin{semilogyaxis}[
	width=.33\textwidth,
    xlabel={dimension $m$},
    ylabel={$\kappa(C_{\rm mono}$)},
    ylabel shift=-5pt,
    ytick={1e-2,1e+4,1e+10,1e+16},
    ymajorgrids,
    yticklabels={},
    xmin=2, xmax=13,
    ymin=1e-2, ymax=1e+16,
]
 
\addplot[name path=upper,color=uppercol,mark=none,ultra thick] table {poly_upperbound.dat};
\addplot[name path=lower,color=lowercol,mark=none, ultra thick] table {poly_lowerbound.dat};
\addplot[shadecol] fill between[of=upper and lower];
\end{semilogyaxis}
\end{tikzpicture}
\begin{tikzpicture}
\pgfplotsset{every axis y label/.append style={yshift=-25pt}}
\begin{semilogyaxis}[
	width=.33\textwidth,
    xlabel={dimension $m$},
    ylabel={$\kappa(C_{\rm gen}$)},
    ytick={1e-2,1e+4,1e+10,1e+16},
    ymajorgrids,
    yticklabel pos=right,
    xmin=2, xmax=25,
    ymin=1e-2, ymax=1e+16,
    yminorticks=false,
]
 
\addplot[name path=upper,color=uppercol,mark=none,ultra thick] table {gen_upperbound.dat};
\addplot[name path=lower,color=lowercol,mark=none, ultra thick] table {gen_lowerbound.dat};
\addplot[shadecol] fill between[of=upper and lower];
\end{semilogyaxis}
\end{tikzpicture}

\caption{
Lower bounds and upper bounds
on the numerical non-identifiability for
the piecewise linear core tensor (left),
  the monomial core tensor (middle) and generic core tensors (right). \label{fig:sings}
}
\end{figure}
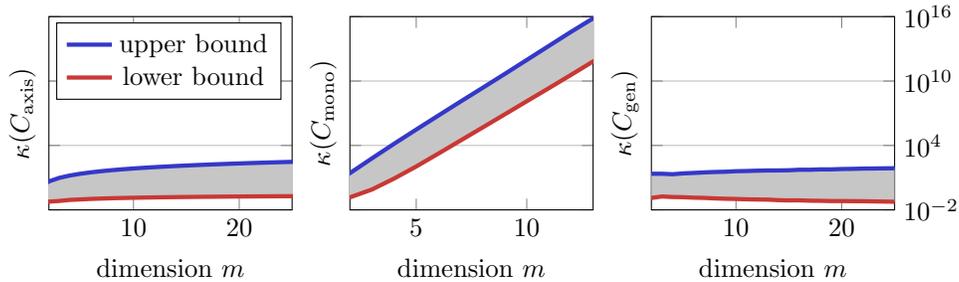

The condition number quantifies 
the numerical identifiability of path recovery for the dictionary with core tensor $C$.
We have derived informative upper and lower bounds on the related notion $\kappa(C)$, in terms of singular values of the flattenings of~$C$.
Figure~\ref{fig:sings} shows these bounds for small $m$.
The lower bound is that in Corollary~\ref{cor:upperbound}.
 The upper bound is that in Proposition \ref{prop:lowerbound}.
 We see that
the numerical non-identifiability of the monomial dictionary grows 
exponentially with $m$. We also empirically observe such a trend for other bases of polynomial functions, such as the Chebyshev functions. On the other hand, the piecewise linear
dictionary is much more stable.
Corollary~\ref{cor:condaxis} shows that the numerical non-identifiability of the piecewise linear dictionary
remains below $6\|C_{\rm axis}\| = 6 \sqrt{\frac{m}{36} + \frac{1}{2}\binom{m}{2} + \binom{m}{3}}$,
as seen on the left in Figure~\ref{fig:sings}.
The numerical non-identifiability of Generic dictionaries seems to
remain below $100$, independently of~$m$. 
The right diagram in  Figure~\ref{fig:sings} shows the average for
 $ 100$ generic signature tensors $C_{\rm gen}$, created  using the first method in Example~\ref{ex:gendic}.
 
We can conclude that for piecewise linear paths, well-conditioned matrices $X$ have signature tensors $[\![C_{\rm axis}; X, X, X]\!]$ that are reasonably far from the non-identifiable locus. 
The same holds for paths from generic dictionaries, in a certain range of $m$.
 However, polynomial paths send well-conditioned matrices $X$ to tensors $[\![C_{\rm mono};X, X, X]\!]$ which are very close to being non-identifiable, even for relatively small values of $m$. This suggests the possibility of numerical challenges for path recovery from such tensors, as we confirm in the
numerical experiments in the next section.

\section{Path Recovery via Optimization}
\label{sec8}

Given a fixed dictionary, our aim is to compute
a path represented by the dictionary whose signature most closely matches an input signature.
 In addition to the issues of numerical identifiability discussed in Section~\ref{sec7},
numerical optimization  has several well-documented drawbacks,
 Since the objective  function $g$ 
 is non-convex, an abundance of local minima can be expected. The problem of local minima is inherent in almost all optimization methods, 
 but
 there are some heuristic ways to overcome the problem. 
 A thorough overview and application of state-of-the-art theory  is out of the scope of this article.
See \cite{NW}.

We performed computational experiments,
for a range of values of $m$ and~$d$. 
We considered
piecewise linear, polynomial, and generic paths, which were created using the first method in Example~\ref{ex:gendic}. 
For each pair $(m,d)$, we generated 100 random matrices~$X \in \RR^{d \times m}$ with entries $x_{ij} \sim N(0,1)$ to represent the path $X\psi$. We computed
$S = \sigma^{(3)}(X)$ up to machine precision and then minimized the function in~\eqref{eq:costf}.

We implemented the Broyden-Fletcher-Goldfarb-Shanno (BFGS) algorithm with an Armijo backtracking line search in Matlab 2018a. This was followed by a trust region Newton method for improved convergence, taken from the Manopt toolbox~\cite{manopt} which allows for direct implementation for matrix inputs.  
We stopped if $\| \nabla g(X) \| < 10^{-10}$ or after 100 steps of the BFGS procedure and 1000 steps of the trust region algorithm. We allowed 10 re-initializations $x_{ij} \sim N(0,1)$ to try to eliminate local minima and other numerical issues that arise from the relatively high degree of the objective function. Let $X^*$ denote the result of this computation. We declare the recovery successful if $\| X^* - X \| / \| X^* \| < 10^{-5}$. Tables~\ref{tab:recovery} and \ref{tab:badrecovery} show the percentages of successful recoveries. The success rate for piecewise linear paths is $100\%$ for small $m$ but it becomes slightly worse for larger $m$.
For paths represented by a generic dictionary, the results are also close to $100\%$.

\begin{table}[h]
\begin{center}
\begin{tabular}{| l | r | r | r | r | r | r | r | r | r | r | r | r | r | r | r | r | r |}
\hline
$ m \backslash d$& 2 & 3 & 4 & 5 & 6 & 7 & 8 & 9 & 10 & 11 & 12 & 13 & 14 & 15  \\ \hline
2 & 100 & 100 & 100 & 100 & 100 & 100 & 100 & 100 & 100 & 100 & 100 & 100 & 100 & 100  \\ \hline
3 & & 100 & 100 & 100 & 100 & 100 &  99 & 100 & 100 & 100 & 100 & 100 & 100 & 100  \\ \hline
4 & & & 100 &  97 & 100 & 100 & 100 & 100 & 100 & 100 & 100 & 100 & 100 & 100  \\ \hline
5 & &   &   &  97 &  99 &  99 & 100 & 100 & 100 & 100 & 100 & 100 & 100 & 100  \\ \hline
6 & &   &   &   &  97 &  95 &  98 &  96 &  97 & 100 & 100 & 100 & 100 & 100  \\ \hline
7 & &   &   &   &   &  91 &  92 &  95 &  96 &  97 &  99 &  99 & 100 & 100  \\ \hline
8 & &   &   &   &   &   &  90 &  92 &  95 &  98 &  99 &  99 &  98 & 100  \\ \hline
9 & &   &   &   &   &   &   &  93 &  90 &  94 &  98 &  95 &  95 &  96  \\ \hline
10 & &   &   &   &   &   &   &   &  85 &  96 &  94 &  97 &  93 &  93 \\ \hline
\end{tabular}\vspace{.5cm}
\begin{tabular}{| l | r | r | r | r | r | r | r | r | r | r | r | r | r | r | r | r | r |}
\hline
$ m \backslash d$& 2 & 3 & 4 & 5 & 6 & 7 & 8 & 9 & 10 & 11 & 12 & 13 & 14 & 15  \\ \hline
2 & 100 & 100 & 100 & 100 & 100 & 100 & 100 & 100 & 100 & 100 & 100 & 100 & 100 &  99 \\ \hline
3 &  &  99 & 100 & 100 & 100 & 100 & 100 &  98 & 100 & 100 & 100 & 100 &  99 & 100  \\ \hline
4 &  &   &  99 &  99 & 100 & 100 & 100 &  98 &  98 &  98 &  99 &  98 &  98 & 100  \\ \hline
5 &  &   &   & 100 & 100 & 100 & 100 & 100 &  99 &  99 & 100 & 100 & 100 & 100  \\ \hline
6 &  &   &   &   &  98 &  99 & 100 & 100 & 100 & 100 &  99 & 100 & 100 &  99  \\ \hline
7 &  &   &   &   &   & 100 &  98 &  97 &  99 &  99 &  99 & 100 & 100 & 100  \\ \hline
8 &  &   &   &   &   &   &  99 &  99 &  99 & 100 &  99 &  98 &  98 &  99  \\ \hline
9 &  &   &   &   &   &   &   &  97 &  92 &  98 &  97 &  97 &  99 &  98  \\ \hline
10 &  &   &   &   &   &   &   &   & 100 &  98 &  97 &  97 & 100 &  99  \\ \hline
\end{tabular}
\end{center}
\caption{\label{tab:recovery} Percentage of successful recoveries for random piecewise linear paths (top)
and random paths represented by generic dictionaries (bottom). 
}
\end{table}


\begin{table}[t]
\begin{center}
\begin{tabular}{| l | r | r | r | r | r | r | r | r | r | r | r | r | r |}
\hline
$ m \backslash d$& 2 & 3 & 4 & 5 & 6 & 7 & 8 & 9 & 10 & 11 & 12 \\ \hline
2 & $100_0$ & $100_0$ & $100_0$ & $100_0$ & $100_0$ & $100_0$ & $100_0$ & $100_0$ & $100_0$ & $100_0$ & $100_0$  \\ \hline
3 & &  $98_2$ &  $99_1$ & $100_0$ & $100_0$ & $100_0$ & $100_0$ & $100_0$ & $100_0$ & $100_0$ & $100_0$ \\ \hline
4 & &   &  $85_{13}$ &  $87_{13}$ &  $94_6$ &  $91_9$ &  $95_5$ &  $98_2$ &  $99_1$ &  $99_1$ &  $98_2$ \\ \hline
5 & &   &   &   $5_{31}$ &  $12_{24}$ &  $20_{30}$ &  $29_{37}$ &  $35_{40}$ &  $47_{41}$ &  $53_{38}$ &  $57_{37}$ \\ \hline
6 & &   &   &   &   $0_1$ &   $0_1$ &   $0_2$ &   $0_3$ &   $0_6$ &   $0_5$ &   $0_9$ \\ \hline
7 & &   &   &   &   &   $0_{21}$ &   $0_{24}$ &  $ 0_{35}$ &   $0_{29}$ &   $0_{28}$ &   $0_{35}$ \\ \hline
\end{tabular}
\end{center}
\caption{\label{tab:badrecovery} The recovery rate for  polynomial paths
is low once the condition number becomes too big.
 Subscripts count the failures due to ill-conditioning.}

\end{table}

Ill-conditioning has occurred when we recover a matrix $X^*$ with a large distance to the original $X$, but whose signatures are the same. We call an instance a {\em failure due to ill-conditioning} if the relative error between the matrices $\| X^* - X \| / \| X^* \|$ exceeds $10^{-5}$ but the relative distance between the signatures is less than $10^{-8}$. This indicates a condition number exceeding 1000.  Such failures never occurred for piecewise linear paths and generic paths, in over 10000 experiments. 
The situation is dramatically worse for polynomial paths: the subscripts in Table~\ref{tab:badrecovery} count the failures due to ill-conditioning. 
For $m \geq 6$ if a matrix with sufficiently close signature was found then in all cases it was a failure due to ill-conditioning. The machine precision inaccuracy in the signature leads to large differences in the recovered matrix. 
The overall recovery rates for polynomial dictionaries are low. We remark that although many of the other failures are not counted as being due to ill-conditioning under our requirements stated above, they often yield a relatively far away matrix with close-by signature tensor.

In conclusion, our experimental findings are 
consistent with the theoretical results on the numerical identifiability in Section \ref{sec7}. We find that
 generic paths and piecewise linear
 dictionaries behave best in numerical algorithms for
 recovering paths.
 The middle diagram
in Figure~\ref{fig:sings} showed that the numerical non-identifiability of the monomial core tensor $C_{\rm mono}$ 
grows rapidly with~$m$. Our experiments confirmed
the difficulty of path recovery from the monomial dictionary.
 
\section{Shortest Paths}
\label{sec10}

Our study has so far been concerned with paths of  low complexity in a space of high dimension.
Such paths are identifiable from their third signature.
In this final section we shift gears. We now come to 
a situation where the number of functions in the dictionary, $m$, is much larger than the dimension of the space,~$d$.
The paths are represented by a dictionary $\psi$, but  identifiability no longer holds for the paths $X \psi$ because there are too
many parameters to recover the matrix $X$ from its third order signature. 
We impose extra constraints to select a meaningful path among those with the same signature. 
A natural constraint is the length of the path. This leads to the problem of
 finding the shortest path for a given signature.

In this section we address the task of computing shortest paths when the
third signature tensor is fixed.
 Recall that the length of a path $\psi : [0,1] \rightarrow \RR^m$ equals
 \begin{equation*}
\mathrm{len}(\psi) \,\,= \,\,\int_0^1 \sqrt{\langle \dot \psi(t), \dot \psi(t) \rangle} \,{\rm d}t,
\end{equation*}
where $\dot \psi(t) = \frac{ d \psi}{dt}$. 
This is a rather complicated function to evaluate in general.
However, things are much easier for piecewise linear paths.
For the $m$-step path  given by the dictionary in~\eqref{PLpath} and 
the matrix $X = (x_{ij})$, the length is given by the formula
\begin{equation*}
\mathrm{len}(X) := \mathrm{len}(X\psi) \,\,=\,\, \sum_{j=1}^m \sqrt{\sum_{i = 1}^d x_{ij}^2}.
\end{equation*}
Note that this function is piecewise differentiable. We can therefore 
regularize the objective function~\eqref{eq:costf} with a length constraint.
This leads to the new function
\begin{equation*}
h(X,\lambda) \,\,=\,\, \mathrm{len}(X) + \lambda g(X),
\end{equation*}
where $\lambda$ is a parameter.
A necessary condition for a minimum is that both the gradient in $X$ and the gradient in $\lambda$ equal zero. The latter requirement ensures that $X$ yields the required signature. A problem with this method is that critical points are usually saddle points, which cannot be easily obtained using standard gradient-related techniques. This holds because 
$h$ is not bounded from below for $\lambda \rightarrow - \infty$.
To work around this, we use a  trick from optimization, see~\cite{NW}.
We fix $\lambda_0$ and minimize
\begin{equation*}
h(X,\lambda_0)/\lambda_0 \,\,=\,\, \lambda_0^{-1} \mathrm{len}(X) + g(X).
\end{equation*}
Once a minimum $X_0$ is found,  we set $\lambda_1 = 2\lambda_0$ and minimize again with $\lambda_1$ and $X_0$ as a starting point for the iteration. We repeat, setting $\lambda_n = 2 \lambda_{n-1}$ until $\lambda_n$ is sufficiently large and the impact of the length constraint is negligible. Then, for some $X_n$, the function $g$ is minimal, i.e.~$X_n$ has the correct signature
up to machine precision. 
Local minima might occur -- a guarantee that $X_n$ gives \emph{the} shortest path
  cannot be made. However, this method has proved to be satisfactory 
  for our application.
 
For $d \leq 3$, the resulting shortest paths corresponding to the piecewise linear dictionary and the monomial dictionary can be easily plotted. We report on two examples,  for $d=2$ and  for $d=3$.
\begin{figure}[h]
\centering \vspace{-0.1in}
\includegraphics[width=.32\textwidth]{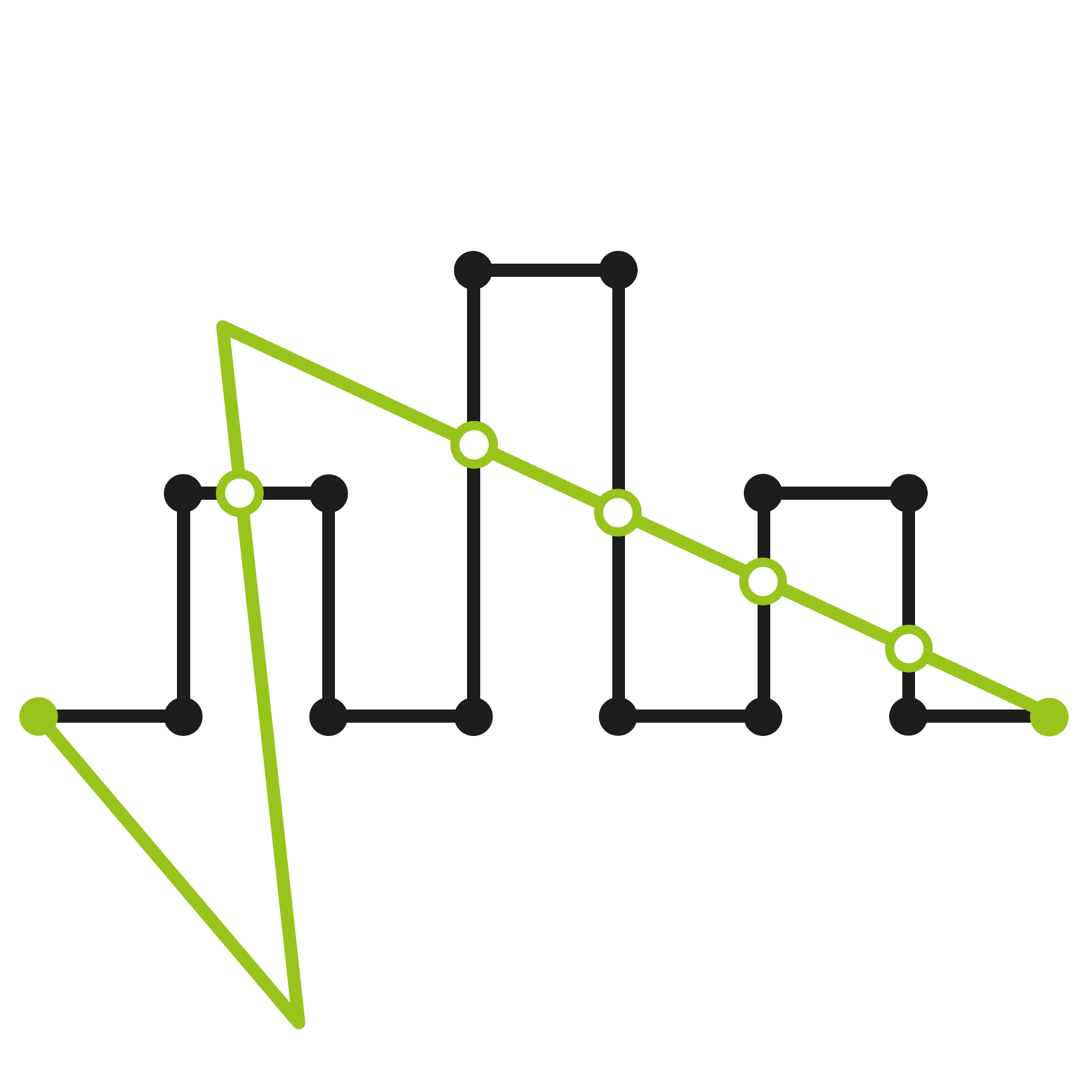}
\includegraphics[width=.32\textwidth]{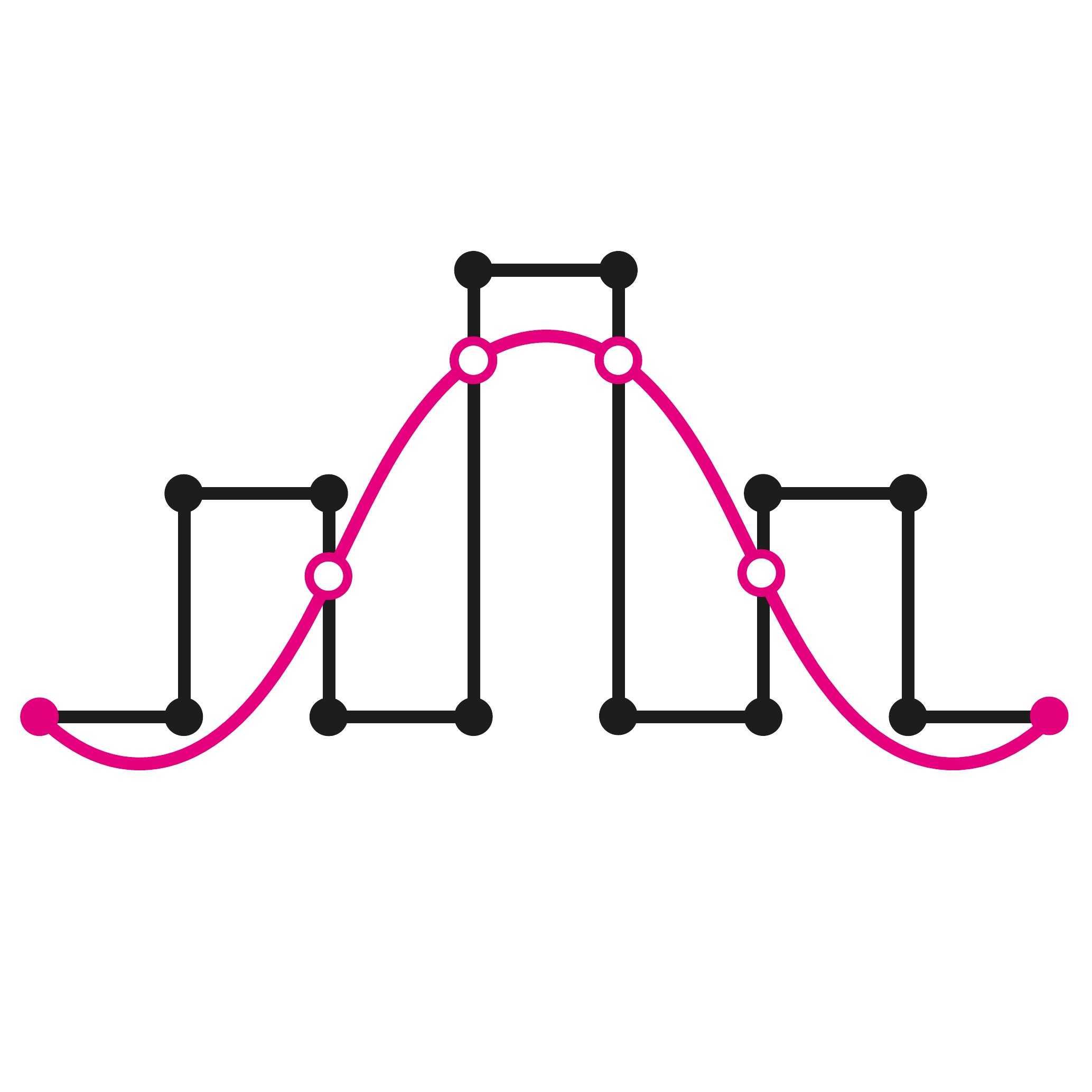}
\includegraphics[width=.32\textwidth]{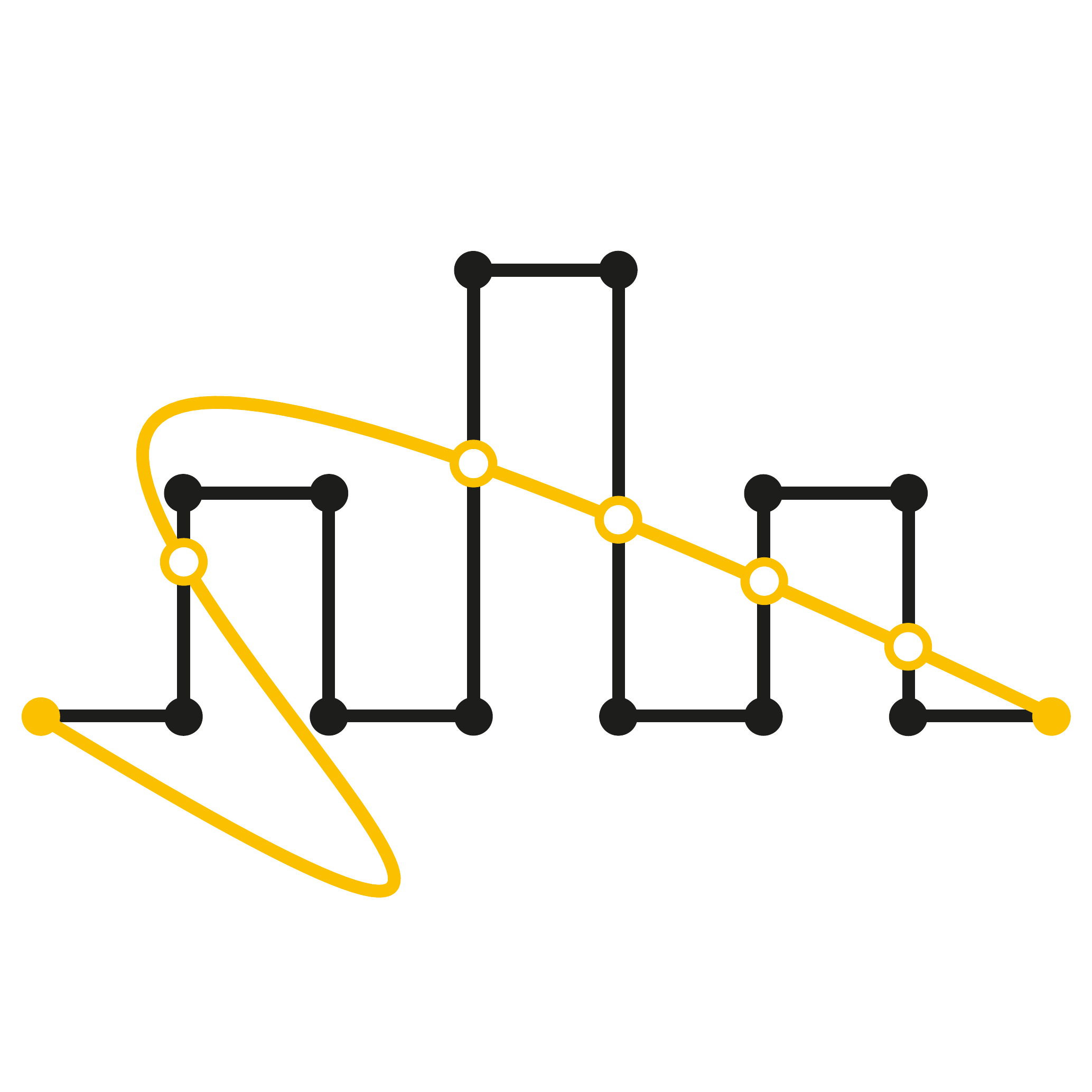} \vspace{-0.12in}
\caption{\label{fig:hist} The shortest path with $m = 3$ steps (left), $m = 100$ steps
(middle), and a polynomial path of degree 3 (right). All have the same third signature as the skyline path.}
\end{figure}

\begin{example}[Skyline path] \rm
Consider the path in Example~\ref{ex:skyline},
with $2 \times 2 \times 2$ signature tensor $S_{\rm skyline}$.
This path is shown in black in Figure~\ref{fig:hist}. For a range of
 values  $m \geq 3$, we computed the shortest piecewise linear path 
 with $m$ steps having signature $S_{\rm skyline}$.
The shortest piecewise linear path with $m = 3$
is depicted in Figure~\ref{fig:hist} on the~left.
 The middle image shows the shortest path with $m = 100$.
 This is an approximation to a shortest smooth path with that signature. We 
 also learned polynomial paths 
 but without length constraints. The right image shows a cubic path.
\end{example}

\begin{example}[$d=3$] \rm
We define the {\em Klee-Minty path} to be 
the following axis path with $7$ steps in $\RR^3$.
  It travels along the edges
 of a $3$-cube and visits all $8$ vertices:
$$ X \,\, = \,\,
\begin{bmatrix}
  \,\,   1 \, &   0  &                   -1   &   \,0 \, &   1  &   \phantom{-}0 &    -1\, \\
 \, \,   0 \, &   1  &  \phantom{-} 0  & \,  0  \,&   0  &                    -1 & \phantom{-}   0 \,\\
  \, \,  0  \, &   0  &  \phantom{-}0   &  \, 1 \,  &  0   &  \phantom{-}0   &  \phantom{-} 0 \,\end{bmatrix}.
 $$
The third order signature tensor of this path equals
\[ S_{\rm km}\,\, = \,\,\, [\![ C_{\rm axis} ; X, X, X ]\!]\,\, = \,\, \,\frac{1}{6}
\left[
\begin{array}{ccc|ccc|ccc}
0 & 0 & 0 & \phantom{-} 0 & 0 & \phantom{-} 0 &\phantom{-}  0 & 6 & 0 \\
0 & 0 & 0 & \phantom{-}  0 & 0 &                  -6 &                    -6 & 3 & 3 \\
0 & 6 & 0 &                    -6 &  3 &                 -3 & \phantom{-} 0 & 0 & 1 \\
 \end{array}
\right] . \]
We expect to find piecewise linear paths with $m = 5$ steps and third
signature $S_{\rm km}$ because $md = 15$ exceeds ${\rm dim}(\mathcal{U}_{3,3}) = 14$. 
Our optimization method found several such paths.
The shortest among them is shown on the left of Figure~\ref{fig:km}.
Next, we computed the shortest path with $m = 100$ steps, shown in the middle of Figure~\ref{fig:km}.

\begin{figure}[htpb]
\centering
\includegraphics[width=.25\textwidth]{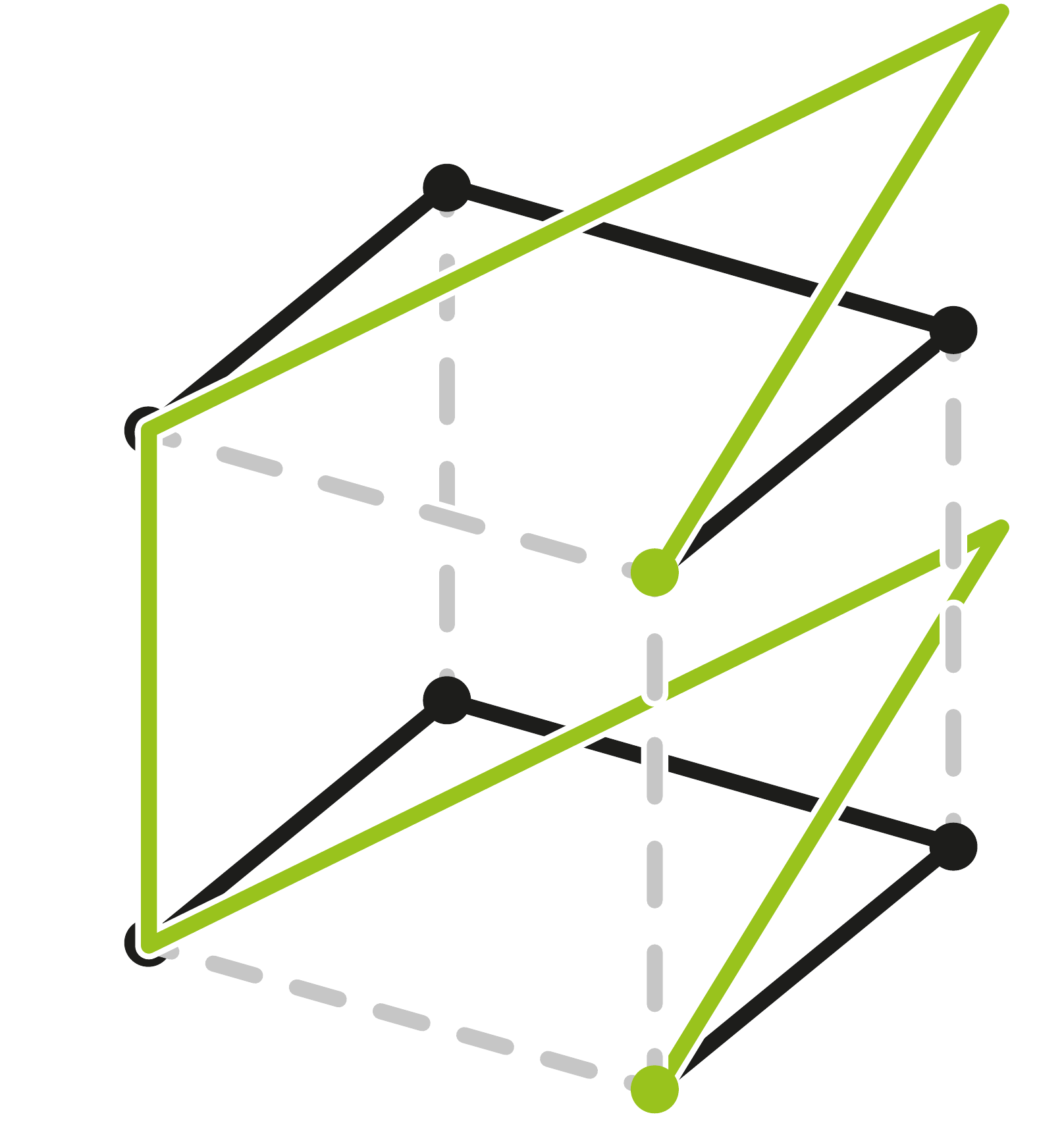} \qquad
\includegraphics[width=.25\textwidth]{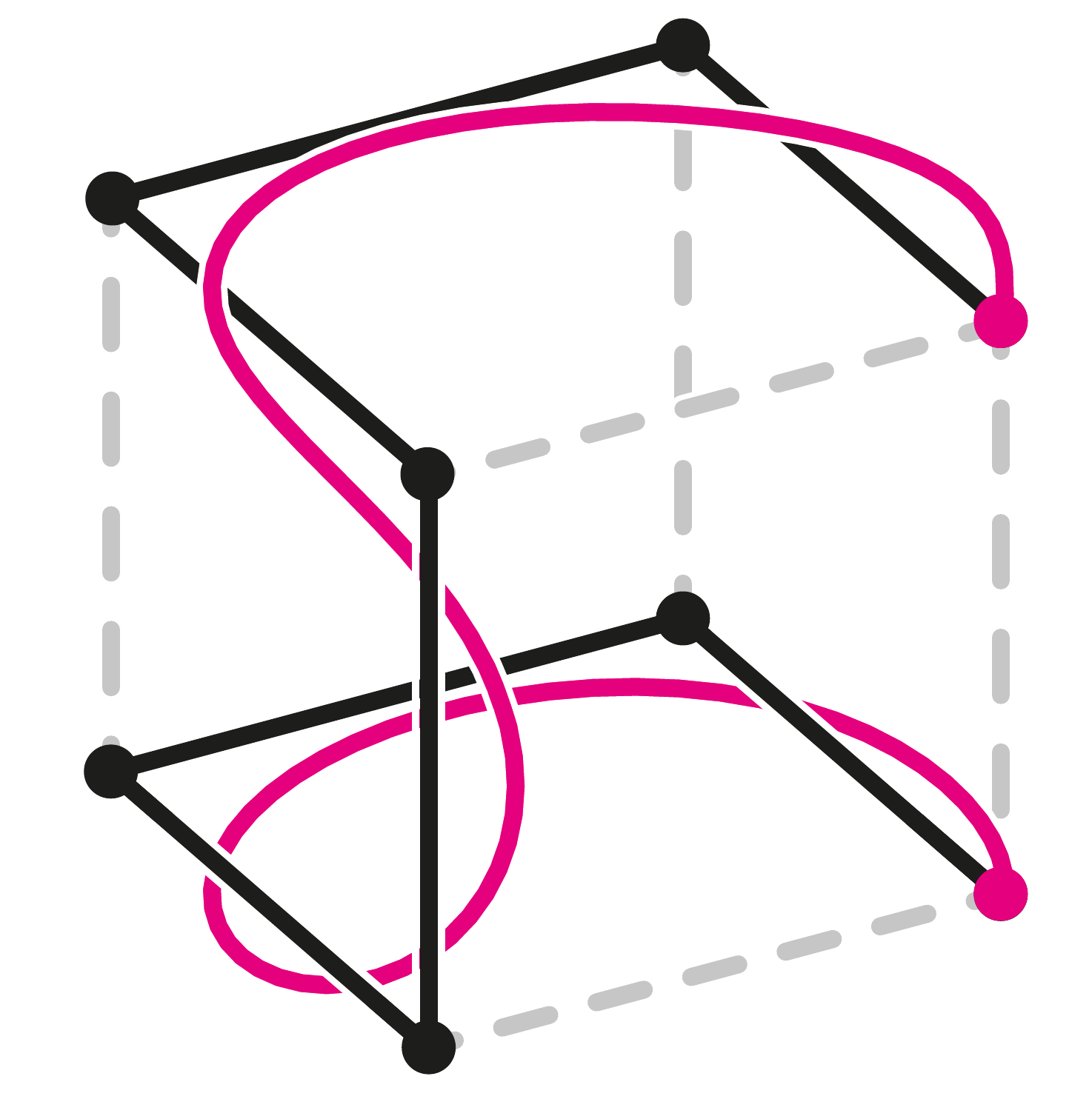} \qquad
\includegraphics[width=.25\textwidth]{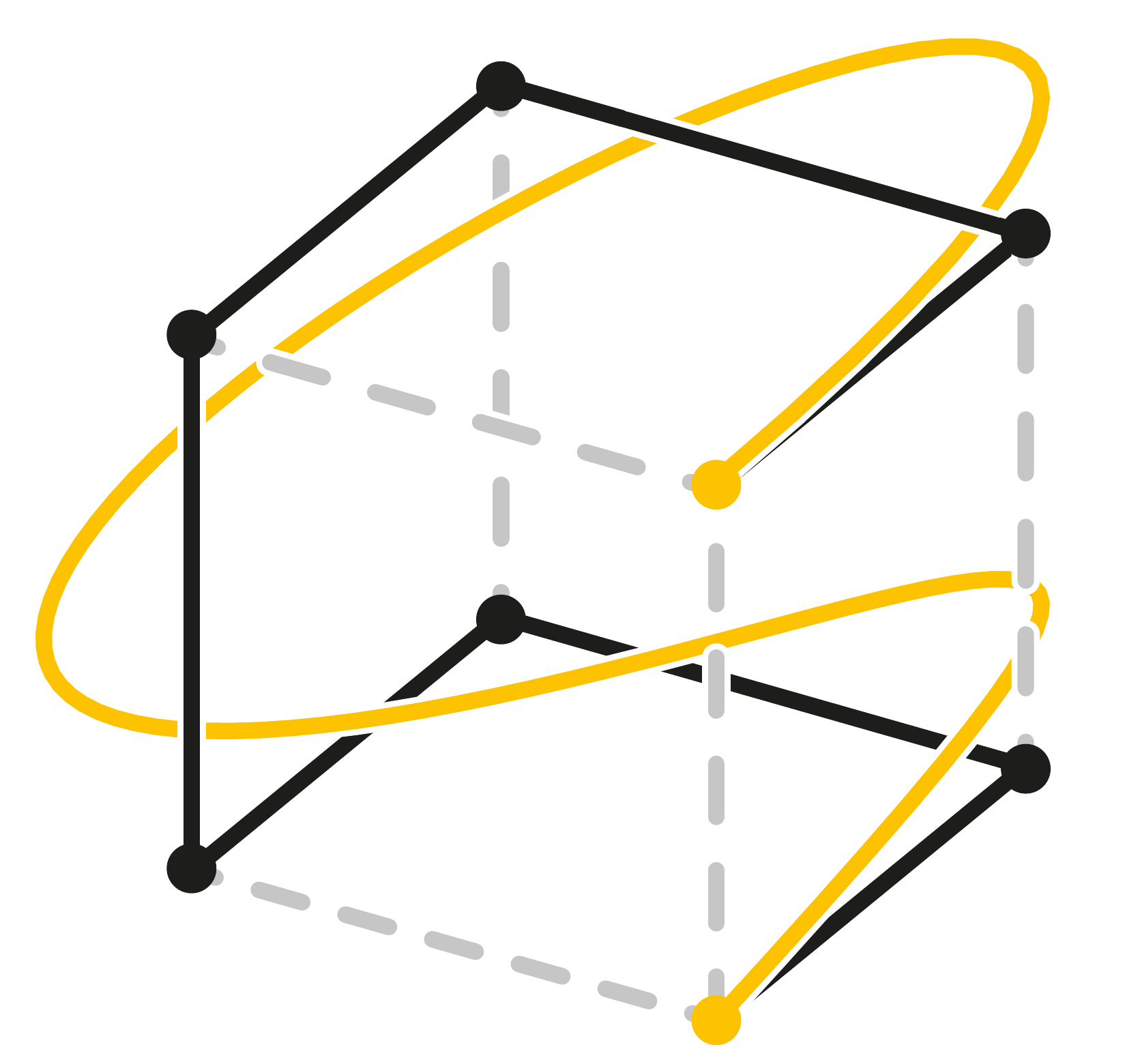}
\caption{The shortest path with $m = 5$ steps
(left) and $m = 100$ steps (middle) 
having the same third order signature as the Klee-Minty path. On the right we see a polynomial path of degree 
$5$ whose third order signature is close to that of the Klee-Minty path. \label{fig:km}}
\end{figure}

Counting dimensions,
 polynomial paths of degree $5$ are expected to fill $\mathcal{U}_{3,3}$.
But we did not find any  path with signature  $S_{\rm km}$. 
A close solution was a matrix $X$~with
\begin{equation*}
\| \,[\![ C_{\rm mono} ; X, X, X ]\!] \,-\, S_{\rm km}\, \|\,\,\, \approx \,\,\, 0.00914 .
\end{equation*}
The associated quintic path $X \psi$ is shown on the right in Figure~\ref{fig:km}.
We believe that the issue is the distinction between the signature image
and the signature variety, discussed in \cite[Section 2.2]{AFS}.
The tensor $S_{\rm km}$ seems to lie in the set
$\mathcal{P}_{3,3,5}^{\RR} \backslash \mathcal{P}_{3,3,5}^{\rm im}$.
\end{example}

  \bigskip 
                                                                                                                                                               
\noindent {\bf Acknowledgments.}
 We thank Robert Guralnick and Vladimir Popov for the
 references used in Theorem \ref{thm:genericpoint}.
 We are grateful to Carlos Am\'endola,
 Paul Breiding,
Laura Colmenarejo, Joscha Diehl,  Francesco Galuppi, Peter B\"urgisser,
 and Andr\'e Uschmajew for helpful conversations.
 Many thanks also to an anonymous referee whose detailed comments 
 led to significant improvements to the paper.
  Anna Seigal and Bernd Sturmfels were partially supported by the  
 US National Science Foundation (DMS-1419018).          

\bigskip
                                                


\bigskip

\noindent
\footnotesize {\bf Authors' addresses:}


\noindent Max Pfeffer, MPI-MiS Leipzig 
\hfill  {\tt max.pfeffer@mis.mpg.de}

\noindent Anna Seigal, UC Berkeley
\hfill {\tt seigal@berkeley.edu}

\noindent Bernd Sturmfels,
 \  MPI-MiS Leipzig and
UC  Berkeley \hfill  {\tt bernd@mis.mpg.de}

\end{document}